\documentclass[11pt]{amsart}
\usepackage{amsfonts}
\usepackage{graphicx}
\usepackage{amssymb}
\usepackage{mathrsfs}
\usepackage{hyperref}

\newtheorem{theorem}{Theorem}[section]
\newtheorem{corollary}[theorem]{Corollary}
\newtheorem{lemma}[theorem]{Lemma}
\newtheorem{proposition}[theorem]{Proposition}
\newtheorem{definition}[theorem]{Definition}
\theoremstyle{definition}
\newtheorem{remark}[theorem]{Remark}

\numberwithin{equation}{section}

\newtheorem{conjecture}{Conjecture}

\newcommand{\beqm}{\begin{eqnarray*}}
\newcommand{\eqm}{\end{eqnarray*}}

\newcommand{\im}{\mathrm{i}}
\newcommand{\Cn}{\mathbb{C}^n}
\newcommand{\C}{\mathbb{C}}
\newcommand{\B}{\mathbb{B}}
\newcommand{\N}{\mathbb{N}}
\newcommand{\fr}{\frac}

\newcommand{\R}{\mathbb{R}}

\newcommand{\Hess}{\mathrm{Hess}_{\mathbb R}}

\DeclareMathOperator{\esssup}{ess\, sup}
\DeclareMathOperator{\IDA}{IDA}
\DeclareMathOperator{\VDA}{VDA}
\DeclareMathOperator{\BDA}{BDA}
\DeclareMathOperator{\BMO}{BMO}
\DeclareMathOperator{\VMO}{VMO}

\DeclareMathOperator{\supp}{supp}
\DeclareMathOperator{\loc}{loc}

\begin{document}
\title{IDA and Hankel operators on Fock spaces}

\author{Zhangjian Hu}
\address{Huzhou University, China}
\email{huzj@zjhu.edu.cn}

\author{Jani A. Virtanen}
\address{University of Reading, England and University of Helsinki, Finland}
\email{j.a.virtanen@reading.ac.uk, jani.virtanen@helsinki.fi}

\begin{abstract}
We introduce a new space IDA of locally integrable functions whose integral distance to holomorphic functions is finite, and use it to completely characterize boundedness and compactness of Hankel operators on weighted Fock spaces. As an application, for bounded symbols, we show that the Hankel operator $H_f$ is compact if and only if $H_{\bar f}$ is compact, which complements the classical compactness result of Berger and Coburn. Motivated by recent work of Bauer, Coburn, and Hagger, we also apply our results to the Berezin-Toeplitz quantization.
\medskip

\noindent\textbf{MSC(2020):} Primary 47B35; Secondary 32A25, 32A37, 81S10\\
\textbf{Keywords:} Fock space, Hankel operator, boundedness, compactness, quantization, $\overline{\partial}$-equation
\end{abstract}

\maketitle

\section{Introduction}
Denote by $L^2$ the Hilbert space of all Gaussian square-integrable functions $f$ on $\C^n$, that is,
$$
	\int_{\C^n} |f(z)|^2 e^{-|z|^2} dv(z) < \infty
$$
where $v$ is the standard Lebesgue measure on $\C^n$. The Fock space $F^2$ (aka Segal-Bargmann space) consists of all holomorphic functions in $L^2$. The orthogonal projection of $L^2$ onto $F^2$ is denoted by $P$ and called the Bergman projection. For a suitable function $f: \C^n\to \C$, the Hankel operator $H_f$ and the Toeplitz operator $T_f$ are defined on $F^2$ by
$$
	H_f = (I-P)M_f \quad{\rm and}\quad T_f  = PM_f.
$$
The function $f$ is referred to as the symbol of $H_f$ and $T_f$. Since $P$ is a bounded operator, it follows that both $H_f$ and $T_f$ are well defined and bounded on $F^2$ if $f$ is a bounded function. For unbounded symbols, despite considerable efforts, see, e.g.~\cite{Ba05, BC94, CHSW, HW18}, characterization of boundedness or compactness of these operators has remained an open problem for more than 20 years.

In this paper, as a natural evolution from $\BMO$ (see~\cite{JN61, Zh12}), we introduce a notion of integral distance to holomorphic (aka analytic) functions $\IDA$ and use it to completely characterize boundedness and compactness of Hankel operators on Fock spaces. Recently, in~\cite{HV22}, which continues our present work, we use $\IDA$ in the Hilbert space setting to characterize the Schatten class properties of Hankel operators. Indeed, the space $\IDA$ is broad in scope, and should have more applications, which we hope to demonstrate in future work in connection with Toeplitz operators.

All our results are proved for weighted Fock spaces $F^p(\varphi)$ consisting of holomorphic functions for which
$$
	\int_{\C^n} |f(z)|^p e^{-p\varphi(z)} dv(z) < \infty,
$$
where $0<p<\infty$ and $\varphi$ is a suitable weight function (see Section~\ref{prelimi} for further details). Obviously, with $p=2$ and $\varphi(z) = \frac\alpha2|z|^2$, we obtain the weighted Fock space $F^2_\alpha$. The study of $L^p$-type Fock spaces was initiated in~\cite{JPR} and has since grown considerably as seen in~\cite{Zh12}.

\medskip

We also revisit and complement a surprising result due to Berger and Coburn~\cite{BC87} which states that for bounded symbols
\begin{equation*}\label{e:phenomenon}
	H_f : F^2\to L^2\ \textrm{is compact}\ \text{if and only if}\ H_{\bar f}\ \textrm{is compact}.
\end{equation*}
In particular, we give a new proof and show that this phenomenon remains true for Hankel operators from $F^p(\varphi)$ to $L^q(\varphi)$ for general weights. What also makes this result striking is that it is not true for Hankel operators acting on other important function spaces, such as Hardy or Bergman spaces.

\medskip

As an application, we will apply our results to the Berezin-Toeplitz quantization, which complements the results in~\cite{BCH18}.

\subsection{Main results}
We introduce the following new function spaces to characterize bounded and compact Hankel operators. Let $0< s\le \infty$ and $0<q< \infty$. For $f\in L^q_{\mathrm {loc}}$, set
 $$
	\big( G_{q,r}(f)(z)\big)^q = \inf_{h\in H(B(z,r))} \frac{1}{|B(z,r)|}\int_{B(z,r)}
	|f-h|^q dv \quad(z\in \C^n)
$$
where $H(B(z,r))$ stands for the set of holomorphic functions in the ball $B(z,r$. We say that $f\in L^q_{\mathrm {loc}}$ is in $\IDA^{s,q}$ if
$$
	\|f\|_{\mathrm{IDA}^{s, q}} =  \|  G_{q, 1}(f)\|_{L^s}<\infty.
$$
We further write $\mathrm{BDA}^q$ for $\IDA^{\infty, q}$ and say that $f\in\mathrm{VDA}^q$ if
$$
	\lim_{z\to \infty} G_{q, 1}(f)(z)=0.
$$
The properties of these spaces will be studied in Section~\ref{space-IDA}.

\medskip

We denote by $\mathcal S$ the set of all measurable functions $f$ that satisfy the condition in~\eqref{e:symbol class}, which ensures that the Hankel operator $H_f$ is densely defined on $F^p(\varphi)$ provided that $0<p<\infty$ and $\varphi$ is a suitable weight. Notice that the symbol class $\mathcal S$ contains all bounded functions. Further, we write $\Hess \varphi$ for the Hessian of $\varphi$ and $\mathrm{E}$ for the $2n\times 2n$ identity matrix---these concepts will be discussed in more detail in Section~\ref{prelimi}. It is important to notice that the condition $\mathrm{Hess}_{\mathbb R}\varphi \simeq \mathrm{E}$ in the following theorems is satisfied by the classical Fock space $F^2$, the Fock spaces $F^2_\alpha$ generated by {\it standard weights} $\varphi(z) = \frac\alpha2|z|^2$ ($\alpha>0$), Fock-Sobolev spaces, and a large class of non-radial weights.

\begin{theorem}\label{main1}
Let $f\in \mathcal S$ and suppose that $\mathrm{Hess}_{\mathbb R}\varphi \simeq \mathrm{E}$ as in~\eqref{weights}.

{\rm (a)} For $0< p\leq q<\infty$ and $q\ge 1$, $H_f : F^p(\varphi)\to L^q(\varphi)$ is bounded if and only if $f\in \BDA^q$, and $H_f$ is compact if and only if $f\in VDA^q$. For the operator norm of $H_f$, we have the estimate
\begin{eqnarray}\label{bounded-g}
        \|H_f\| \simeq \left\|   f \right\|_{\mathrm{BDA}^q}.
\end{eqnarray}

{\rm (b)} For $1\le q<p <\infty$, $H_f: F^p(\varphi)\to L^q(\varphi)$ is bounded if and only if it is compact, which is equivalent to $f\in \mathrm{IDA}^{s, q}$,
where $s={\fr {pq}{p-q}}$, and
\begin{equation}\label{q-lessthen-p-a}
	\|H_f\| \simeq \| f  \|_{\IDA^{s, q}}.
\end{equation}

{\rm (c)} For $0<p\le q\le 1$ and $f\in L^\infty$, $H_f : F^p(\varphi)\to L^q(\varphi)$ is bounded with
\begin{equation}\label{q-is-small}
	\|H_f\|\le C \|f\|_{L^\infty}
\end{equation}
and compact if and only if $f\in \VDA^q$.
\end{theorem}

We first note that Theorem~\ref{main1} is new even for Hankel operators acting from $F^2$ to $L^2$. Previously only characterizations for $H_f$ and $H_{\bar f}$ to be {\it simultaneously} bounded (or \textit{simultaneously} compact) were known. These were given in terms of the bounded (or vanishing) mean oscillation of $f$ by Bauer \cite{Ba05} for $F^2$ and by Hu and Wang~\cite{HW18} for Hankel operators from $F^p_\alpha$ to $L^q_\alpha$. In Theorem~\ref{doubling-bdd} of Section~\ref{Further remarks}, we obtain these results as a simple consequence of Theorem~\ref{main1}. We also mention our recent work~\cite{HV22}, which gives a complete characterization of Schatten class Hankel operators.

Theorem \ref{main1} should also be compared with the results for Hankel operators on Bergman spaces $A^p$. Indeed, characterizations for boundedness and compactness can be found in~\cite{Ax86} for anti-analytic symbols, in~\cite{HV19} for bounded symbols, and in~\cite{HL19, Li94, Lu92, PZZ16} for unbounded symbols. What makes the study of the two cases different are the properties such as $F^p\subset F^q$ for $p\le q$ (as opposed to $A^q\subset A^p$) and certain nice geometry on the boundary of these bounded domains, which in turn helps with the treatment of the $\bar \partial$-problem.

What is very different about the results on Hankel operators acting on these two types of spaces is that our next result is only true in Fock spaces (see~\cite{HV19} for an interesting counterexample for the Bergman space).

\begin{theorem}\label{main2}
Let $f\in L^\infty$ and suppose that $\mathrm{Hess}_{\mathbb R}\varphi \simeq \mathrm{E}$ as in~\eqref{weights}. If $0<p\le q<\infty$ or $1\le q < p<\infty$, then $H_f : F^p(\varphi) \to L^q(\varphi)$ is compact if and only if $H_{\bar f}$ is compact.
\end{theorem}

For Hankel operators on the Fock space $F^2$, Theorem~\ref{main2} was proved by Berger and Coburn~\cite{BC87} using $C^*$-algebra and Hilbert space techniques and by Stroethoff~\cite{St92} using elementary methods. More recently in~\cite{HV19}, limit operator techniques were used to treat the reflexive Fock spaces $F^p_\alpha$. However, our result is new even in the Hilbert space case because of the more general weights that we consider. As a natural continuation of our present work, in~\cite{HV22}, we prove that, for $f\in L^\infty$, the Hankel operator $H_f$ is in the Schatten class $S_p$ if and only if $H_{\bar f}$ is in the Schatten class $S_p$ provided that $1<p<\infty$.

As an application and further generalization of our results, in Section~\ref{quantization}, we provide a complete characterization of those $f\in L^\infty$ for which
\begin{equation}\label{e:R2}
	\lim_{t\to 0} \|T^{(t)}_f T^{(t)}_g - T^{(t)}_{fg}\|_t = 0
\end{equation}
for all $g\in L^\infty$, where $T^{(t)}_f = P^{(t)}M_f : F^2_t(\varphi) \to F^2_t(\varphi)$ and $P^{(t)}$ is the orthogonal projection of $L^2_t(\varphi)$ onto $F^2_t(\varphi)$. Here $L^2_t = L^2(\C^n, d\mu_t)$ and
$$
	d\mu_t (z)= \fr 1{t^n} \exp\left \{-  2\varphi\left( \fr    z { \sqrt{t}}\right)\right \} dv(z).
$$
The importance of the semi-classical limit in~\eqref{e:R2} stems from the fact that it is one of the essential ingredients of the deformation quantization of Rieffel~\cite{Ri89, Ri90} in mathematical physics. Our conclusion related to (\ref{e:R2}) extends and complements the main result in~\cite{BCH18}.

\subsection{Approach}
A careful inspection shows that the methods and techniques used in~\cite{ BC86, BC87,  HV19,  PSV14, St92} depend heavily upon the following three aspects. First, the explicit representation of the Bergman kernel $K(z, w)$ for standard weights $\varphi(z)= \fr \alpha 2 |z|^2$ has the property that
\begin{equation}\label{classical}
  K(z,w) e^{-\fr \alpha 2 |z|^2 -\fr \alpha 2 |w|^2 } = e^{\fr \alpha 2 |z-w|^2}.
\end{equation}
However, for the class of weights we consider, this quadratic decay is known not to hold (even in dimension $n=1$), and is expected to be very rare \cite{Ch91}. The second aspect involves the  Weyl unitary operator $W_a$ defined as
$$
     W_af= f\circ \tau_a k_a,
 $$
where $\tau_a$ is the translation by $a$ and $k_a$ is the normalized reproducing kernel. As a unitary operator on $F^p_\alpha$ (or on $L^p_\alpha$),  $W_a$ plays a very important role in the theory of the Fock spaces $F^p_\alpha$ (see \cite{Zh12}). Unfortunately, no analogue of Weyl operators is currently available for $F^p(\varphi)$ when $\varphi\neq \fr \alpha 2 |w|^2 $. The third aspect we mention is Banach (or Hilbert) space techniques, such as the adjoint (for example, $H_f^*$) and the duality. However, when $0<p<1$,  $F^p(\varphi)$ is only an $F$-space (in the sense of~\cite{Rudin}) and the usual Banach space techniques can no longer be applied.

To overcome the three difficulties mentioned above, we introduce function spaces $\IDA$, $\BDA$ and $\VDA$, and develop their theory, which we use to characterize those symbols $f$ such that $H_f$ are bounded (or compact) from $ F^p(\varphi)$ to  $L^q(\varphi)$. Our characterization of the boundedness of $H_f$ extends the main results of~\cite{Ba05, HW18, PSV14}. It is also worth noting that as a natural generalization of $\BMO$, the space $\IDA$ will have its own interest and will likely be useful to study other (related) operators (such as Toeplitz operators).

In our analysis, we appeal to the $\overline \partial$-techniques several times.  As the canonical solution to $\overline \partial u = g \partial f$, $H_f g$ is naturally connected with the $\overline \partial$-theory. H\"{o}rmander's theory provides us with the $L^2$-estimate, but less is known about $L^p$-estimates on $\Cn$ when $p\neq 2$. With the help of a certain auxiliary integral operator, we obtain $L^p$-estimates of the Berndtsson-Anderson's solution~\cite{BA82} to $\overline \partial$-equation.  Our approach to handling weights whose curvature is uniformly comparable to the Euclidean metric form is similar to the treatment in~\cite{SV12} which was initiated by Berndtsson and Ortega-Cerd\`{a} in~\cite{BO95}, and a number of the techniques we use here were inspired by this approach. Although the work in~\cite{BO95} is restricted to $n=1$, some of the results were extended by Lindholm to higher dimensions in~\cite{Li01}, and the others are easy to modify.

The outline of the paper is  as follows. In Section~\ref{prelimi} we study preliminary results on the Bergman kernel which are needed throughout the paper, and we also establish estimates for the  $\overline \partial$-solution developed in~\cite{BA82}. In Section~\ref{space-IDA}, a notion of  function spaces $\mathrm{IDA}^{s, q}$ is introduced.   We obtain a useful decomposition for functions in $\mathrm{IDA}^{s, q}$ (compare with the decompositions of $\BMO$ and $\VMO$). Using this decomposition, we obtain the completeness of $\IDA^{s, q}/H(\Cn)$  in $\|\cdot\|_{\mathrm{IDA}^{s, q}}$. In Sections~\ref{hankel-A} and~\ref{bdd-symbols} we prove Theorems~\ref{main1} and~\ref{main2}, respectively. For the latter theorem, we also appeal to the Calder\'on-Zygmund theory of singular integrals, and in particular employ the Ahlfors-Beurling operator to obtain certain estimates on $\partial$ and $\overline\partial$ derivatives. In Section~\ref{quantization}, we present an application of our results to quantization. In the last section, we give further remarks together with two conjectures.

Throughout the paper, $C$ stands for positive constants which may  change from line to line, but does not depend on functions being considered. Two quantities $A$ and $B$ are called equivalent, denoted by $A\simeq B$, if there exists some $C$ such that $C^{-1} A \leq B \leq C A$.

\section{Preliminaries}\label{prelimi}
Let $ {\C}^n={\mathbb R}^{2n}$ be the $n$-dimensional complex Euclidean space and denote by  $v$ the Lebesgue measure on ${\C}^n$. For $z=(z_1, \cdots, z_n)$ and $w=(w_1, \cdots,
w_n)$ in ${\C}^n$, we write $z\cdot {\overline w} =z_1\overline{w}_1+\cdots+z_n\overline{w}_n$  and $|z|=\sqrt{z\cdot {\overline z} }$. Let $H({\C}^n)$ be the family of all holomorphic functions on ${\C}^n$. Given a domain $\Omega$ in $\Cn$ and a positive Borel measure $\mu$ on $\Omega$, we denote by $L^p(\Omega, d\mu)$ the space of all Lebesgue measurable functions $f$ on $\Omega$ for which
$$
	\|f\|_{L^p(\Omega, d\mu)}= \left\{\int_\Omega |f|^pd\mu\right\}^{\fr 1p}<\infty
	\,\, \textrm{ for } \,\, 0<p< \infty
$$
and  $\|f\|_{L^\infty(\Omega, dv)}=\esssup_{z\in \Omega} |f(z)|<\infty$ for $p=\infty$. For ease of notation, we simply write $L^p$ for the space $L^p(\C^n, dv)$.

\subsection{Weighted Fock spaces}\label{weighted Fock spaces}
For a real-valued weight $\varphi\in C^2({\mathbb C}^{n})$ and $0<p<\infty$, denote by
$L^{p}(\varphi)$ the space $L^p(\Cn, e^{-p\varphi}dv)$ with norm
$\|\cdot\|_{p, \varphi}= \|\cdot\|_{L^p(\Cn, e^{-p\varphi}dv)}$. Then the Fock space $F^p(\varphi)$ is defined as
$$
	F^{p}(\varphi)=L^{p}(\varphi)\cap H({\C}^n)
$$
and
$$
  F^{\infty}(\varphi)=\left\{f\in H({\C}^n):\|f\|_{\infty, \varphi}= \sup\limits_{z\in {\C}^n} |f(z)|e^
{-\varphi(z)}<\infty\right\}.
$$
For $1\leq p\leq\infty$,  $F^{p}(\varphi)$ is a Banach space in the norm $\|\cdot\|_{p, \varphi}$ and $F^2(\varphi)$ is a Hilbert space.  For $0<p<1$, $F^{p}(\varphi)$ is an $F$-space with metric given by $d(f,g)=\|f-g\|^p_{p, \varphi}$.

Other related and widely studied holomorphic function spaces include the Bergman spaces $A^p_\alpha(\B^n)$ of the unit ball $\B^n$ consisting of all holomorphic functions $f$ in $L^p(\B^n, dv_\alpha)$, where $0<p<\infty$, $dv_\alpha(z) = (1-|z|^2)^\alpha\, dv(z)$ and $\alpha>-1$.

In this paper we are interested in Fock spaces $F^p(\varphi)$ with certain uniformly convex weights $\varphi$. More precisely,  suppose $\varphi= \varphi (x_1, x_2,  \cdots, x_{2n}) \in   C^2( {\mathbb R}^{2n})$ is real valued, and there are positive constants $m$ and $M$ such that  $\mathrm{Hess}_{\mathbb R}\varphi$, the real Hessian, satisfies
\begin{equation}\label{weights}
	m{\mathrm E}  \le  \mathrm{Hess}_{\mathbb R}\varphi(x)
	= \left( \fr {\partial^2 \varphi(x)}{\partial x_j \partial x_k} \right)_{j, k=1}^{2n}\le M {\mathrm E}
\end{equation}
where ${\mathrm E}$ is the $2n\times 2n$ identity matrix; above, for symmetric matrices $A$ and $B$, we used the convention that $A \le B$ if $B-A$ is positive semi-definite. When \eqref{weights} is satisfied, we write $\mathrm{Hess}_{\mathbb R}\varphi \simeq {\mathrm E}$. A typical model of such weights is given by $\varphi(z)=\frac{\alpha }{2}|z|^2$ for $z=(z_1, z_2, \cdots, z_n)$ with $z_j=x_{2j-1} +\mathrm{i} x_{2j}$, which induces the weighted Fock space $F^p_\alpha$ studied by many authors (see, e.g.,~\cite{Zh12}). Another popular example is $\varphi(z)= |z|^2 -\fr 12\log (1+|z|^2)$, which gives the so-called Fock-Sobolev spaces studied for example in~\cite{CZ12}. Notice that the weights $\varphi$ satisfying~\eqref{weights} are not only radial functions as the example $\varphi(z)= |z|^2 + \sin[(z_1+\overline z_1)/2]$ clearly shows.

For $x=(x_1, x_2,  \cdots, x_{2n}),  \ t=(t_1, t_2, \cdots, t_{2n})\in {\mathbb R}^{2n}$, write
$ z_j=x_{2j-1} + \mathrm i x_{2j } $, $ \xi_j=t_{2j-1} + \mathrm i t_{2j } $ and $\xi=(\xi_1, \xi_2, \cdots, \xi_n)$. An elementary calculation similar to that on page~125 of~\cite{Kr82} shows
\begin{align*}
   \mathrm {Re} \sum_{j, k=1}^{ n} \fr { \partial ^2 \varphi}{\partial z_j {  \partial} z_k}(z)  \xi_j {\xi}_k
  +  \sum_{j, k=1}^{ n} \fr { \partial ^2 \varphi}{\partial z_j { \partial} \overline { z}_k}(z)  \xi_j \overline{\xi}_k
	&=\frac12 \sum_{j, k=1}^{2n} \fr { \partial ^2 \varphi}{\partial x_j \partial x_k}(x)  t_j t_k\\ 	&\ge  \fr 12  m|\xi|^2.
\end{align*}
Replacing $\xi$ with ${\mathrm i}\xi$ in the above inequality gives
$$
	-\mathrm {Re} \sum_{j, k=1}^{ n} \frac{ \partial ^2 \varphi}{\partial z_j {  \partial} z_k}(z)  \xi_j {\xi}_k +  \sum_{j, k=1}^{ n} \frac{\partial ^2 \varphi}{\partial z_j { \partial} \overline { z}_k}(z)  \xi_j \overline{\xi}_k
	\ge \frac12  m|\xi|^2.
$$
Thus,
$$
	\sum_{j, k=1}^{n} \frac{\partial ^2 \varphi}{\partial z_j { \partial} \overline { z}_k}(z)  \xi_j \overline{\xi}_k
	\ge  \frac12 m|\xi|^2.
$$
Similarly, we have an upper bound for the complex Hessian of $\varphi$. Therefore,
$
   m  \omega_0 \le dd^c \varphi \le M \omega_0
$,
where $\omega_0=dd^c |z|^2$ is the
Euclidean K\"{a}hler form on ${\C}^n$ and  $d^c=\frac{\sqrt{-1}}{4}\left(\overline{\partial}-\partial\right)$.  This implies that the theory  in \cite{SV12} and \cite{HL14} is applicable in the present setting.

For $z\in{{\C}}^n$ and $r>0$, let $B(z, r)=\left\{w\in{{\C}}^n: \left|w-z\right|<r\right\}$ be the ball with center at $z$ with radius $r$. For the proof of the following weighted Bergman inequality, we refer to Proposition~2.3 of~\cite{SV12}.

\begin{lemma}\label{submean-value}
Suppose $0<p\le \infty$. For each $r>0$ there is some $C>0$ such that if $f\in F^p(\varphi)$ then
$$
   \left| f(z) e^{-\varphi(z)}\right|^p \le C \int_{B(z, r)} \left| f(\xi) e^{-\varphi(\xi)}\right|^p dv(\xi).
$$
\end{lemma}

It follows from the preceding lemma that  $\|f\|_{q, \varphi}\le C\|f\|_{p, \varphi}$ and
\begin{equation}\label{nest-prop}
 F^p(\varphi)  \subseteq F^q(\varphi) \,\, \textrm{ for }\, 0<p\le q\le \infty.
\end{equation}
This inclusion  is completely different from that of the Bergman spaces.

 \begin{lemma}\label{basic-est}
 There exist positive constants $\theta$ and $C_1$, depending only on $n$, $m$ and $M$ such that
\begin{equation}\label{basic-est-a}
\left|K(z,w)\right|\leq  C_1 e^{ \varphi(z)+\varphi(w)} e^{-\theta|z-w|}\,  \textrm{ for all }\, z, w\in{{\C}}^n,
\end{equation}
and  there exists positive constants $C_2$ and $r_0 $ such that
\begin{equation}\label{basic-est-b}
\left|K(z,w)\right|\geq C_2 e^{ \varphi(z)+\varphi(w)}
\end{equation}
for   $z\in{{\C}}^n$ and  $w\in B\left(z, r_0\right)$.
\end{lemma}

The estimate \eqref{basic-est-a}~appeared in \cite{Ch91} for $n=1$ and in~\cite{De98} for $n\ge 2$, while the inequality \eqref{basic-est-b} can be found in~\cite{SV12}.

For  $z\in {{\C}}^n$, write $k_{ z}(\cdot)=\frac{K(\cdot,z)}{\sqrt{K(z,z)}}$ for the normalized Bergman kernel.  Then Lemma \ref{basic-est} implies that
\begin{equation}\label{nor-bergman}
 \fr 1 C e^{ \varphi(z)} \le  \|K(\cdot,z)\|_{p, \varphi}  \le C  e^{ \varphi(z)} \, \textrm{ and }\,   \fr 1 C \le \|k_{z}\|_{p,\varphi}\le C, \, \textrm{ for }\, z\in \Cn,
 \end{equation}
and $\lim_{|z|\to \infty} k_z(\xi)=0$ uniformly in $\xi$ on compact subsets of $\Cn$.

\subsection{The Bergman projection}
For Fock spaces, we denote by $P$ the orthogonal projection of $L^2(\varphi)$ onto $F^2(\varphi)$, and refer to it as the Bergman projection. It is well known that $P$ can be represented as an integral operator
\begin{equation}\label{projection-a}
	Pf(z)=\int_{{\C}^n}K(z,w)f(w)e^{-2\varphi(w)}dv(w)
\end{equation}
for $z\in \C^n$, where $K (\cdot, \cdot)$ is the Bergman (reproducing) kernel of $F^2(\varphi)$.

As a consequence of Lemma~\ref{basic-est}, it follows that the Bergman projection $P$ is bounded on $L^p(\varphi)$ for $1\le p\le \infty$, and $P|_{F^p(\varphi)}= \mathrm{I}$ for $0< p\le \infty$; for further details, see Proposition~3.4 and Corollary~3.7 of~\cite{SV12}.

\subsection{Hankel operators}
To define Hankel operators with unbounded symbols, consider
$$
   \Gamma  =\left\{ \sum_{j=1}^N a_j K(\cdot, {z_j}): N\in \mathbb{N}, a_j\in  \mathbb{C}, z_j\in \Cn,  \textrm{ for } 1\le j\le N \right\},
$$
and the symbol class
\begin{equation}\label{e:symbol class}
     \mathcal S=\left \{f \textrm{ measurable on }\Cn: f g\in L^1(\varphi) \textrm{ for  } g \in \Gamma\right \}.
\end{equation}
Given $f \in \mathcal{S}$, the Hankel operator $H_f = (\mathrm I-P)M_f$ with symbol $f$ is well defined on $\Gamma$. According to Proposition~2.5 of~\cite{HV20}, for $0<p<\infty$, the set $\Gamma $ is dense in $F^p(\varphi)$, and hence the Hankel operator $H_f$ is densely defined on $F^p(\varphi)$.

\subsection{Lattices in $\C^n$}\label{lattices} Given   $r>0$,  a  sequence
$\{a_k\}_{k=1}^{\infty}$ in ${{\C}}^n$  is called  an $r$-lattice if the balls $\left\{B(a_k, r)\right\}_{k=1}^{\infty}$ cover ${{\C}}^n$ and $\left\{B\left(a_k, \frac{r}{2\sqrt{n}}\right)\right\}_{k=1}^{\infty}$ are pairwise disjoint. A typical model of an $r$-lattice is the sequence
\begin{equation}
	\left \{\frac{r}{\sqrt n}(m_1+k_1\im, m_2+k_2\im,\cdots, m_n+k_n\im)\in \Cn : m_j, k_j\in \mathbb{Z}, j=1, 2, \cdots, n \right\}.
\end{equation}

Notice that there exists an
integer $N$ depending only on the dimension of ${\Cn}$ such that, for any $r$-lattice $\{a_k\}_{k=1}^{\infty}$,
\begin{equation}\label{lattice}
1 \leq \sum\limits_{k=1}^{\infty}\chi_{B(a_k,2r)}(z) \leq N
\end{equation}
for $z\in\C^n$, where $\chi_E$ is the characteristic function of $E\subset \Cn$. These well known facts are explained in~\cite{Zh12} when $n=1$ and they can be easily generalized to any $n\in\N$.

\subsection{Fock Carleson measures} In the theory of Bergman spaces, Carleson measures provide an essential tool for treating various problems, especially in connection with bounded operators, functions of bounded mean oscillation, and their applications; see, e.g.~\cite{Zh05}. In Fock spaces, Carleson measures play a similar role---see~\cite{Zh12} for the Fock spaces $F^p_\alpha$. Carleson measures for Fock-Sobolev spaces were described in~\cite{CZ12}. In~\cite{SV12}, Carleson measures for generalized Fock spaces (which include the weights considered in the present work) were used to study bounded and compact Toeplitz operators. Finally, their generalization to $(p,q)$-Fock Carleson measures was carried out in~\cite{HL14}, which is indispensable to the study of operators between distinct Banach spaces and will be applied to analyze Hankel operators acting from $F^p(\varphi)$ to $L^q(\varphi)$ in our work.

We recall the basic theory of these measures. Let $0<p, q<\infty$ and let  $\mu\geq 0$ be a positive Borel measure on $\Cn$. We
call  $\mu$   a $(p, q)$\textit{-Fock Carleson measure}  if the embedding
 $\mathrm{I}: F^{p}(\varphi)\rightarrow L^{q}(\Cn, e^{-q\varphi}d\mu)$ is
bounded. Further, the measure $\mu$ is referred to as a vanishing $(p,q)$-Fock Carleson measure if in addition
$$
	\lim\limits_{j\rightarrow\infty}\int_{{{\C}}^n}\left|f_j(z)e^{-\varphi(z)}\right|^q d\mu(z)=0
$$
whenever $\{f_j\}_{j=1}^{\infty}$ is  bounded   in $F^{p}(\varphi)$ and
converges to $0$ uniformly on  any compact subset of ${{\C}}^n$ as
$j\to \infty$. Fock Carleson measures have been completely characterized in~\cite{HL14} and we only add the following simple result, which is trivial for Banach spaces and can be easily proved in the other cases.

\begin{proposition}\label{F-C-cpt}
Let $0<p,q<\infty$ and $\mu$ be a positive Borel measure on $\C^n$. Then $\mu$ is a vanishing $(p,q)$-Fock Carleson measure if and only if the inclusion map $\mathrm{I}$ is compact from $F^p(\varphi) \to L^q(\C^n, d\mu)$.
\end{proposition}
\begin{proof}
It is not difficult to show that the image of the unit ball of $F^p(\varphi)$ under the inclusion is relatively compact in $L^q(\C^n, e^{q\varphi}d\mu)$. We leave out the details.
\end{proof}

\subsection{Differential forms and an auxiliary integral operator} As in \cite{Kr82}, given two nonnegative integers $ s, t\le n $, we write
\begin{equation}\label{p-q-form}
   \omega= \sum_{|\alpha|=s, |\beta|=t} \omega_{\alpha, \beta} dz^\alpha\wedge
     d{\overline z}^\beta
\end{equation}
for a differential form of type $(s, t)$. We denote by $L_{s, t}$ the family of all $(s, t) $-forms $\omega$ as in (\ref{p-q-form}) with coefficients $\omega_{\alpha, \beta}$ measurable on  $\Cn$ and set
\begin{equation}\label{norm of differential form}
  	|\omega|= \sum_{|\alpha|=s, |\beta|=t} |\omega_{\alpha, \beta}|\ {\rm and}\
	\|\omega\|_{p, \varphi}= \left \||\omega|\right\|_{p, \varphi}.
\end{equation}

Given a weight function $\varphi$ satisfying (\ref{weights}),  we define   an integral operator $A_\varphi$  as
\begin{multline}\label{A-operator}
      A_\varphi(\omega )(z)= \int_{\Cn} e^{\langle 2 \partial \varphi, z-\xi  \rangle }\\
      \times\sum_{j<n } \omega(\xi)\wedge
 \fr{\partial|\xi-z|^2\wedge (2 {\overline \partial} \partial\varphi(\xi))^j \wedge ( {\overline \partial} \partial |\xi-z|^2)^{n-1-j} }{j!|\xi-z|^{2n-2j}}\,
\end{multline}
 for  $\omega\in L_{0, 1}$, where $\langle  \partial \varphi(\xi), z-\xi  \rangle  = \sum_{j=1}^n \fr {\partial \varphi}{{\partial}{ \xi}_j}(\xi)(z_j-\xi_j)$   as denoted on page~92 in~\cite{BA82}.

For $(s_1, t_1)$-form $\omega_A$ and $(s_2, t_2)$-form $\omega_B$ with $s_1+s_2\le n, t_1+t_2\le n$,  it is easy to verify that
 $ \left|\omega_A \wedge\omega_B  \right|\le \left|\omega_A\right| \left|\omega_B\right|$. Therefore, for the $(n, n)$-form inside the integral of the right hand side of (\ref{A-operator}),  we obtain
$$
	\left|  \omega(\xi)\wedge\frac{\partial|\xi-z|^2\wedge (2 {\overline \partial} \partial\varphi)^j \wedge ( {\overline \partial} \partial |\xi-z|^2)^{n-1-j} }{j!|\xi-z|^{2n-2j}}\right| 
	\le C \fr {|\omega(\xi)|}{|\xi-z|^{2n-2j-1}}
$$
because $\mathrm i  {\partial \overline \partial}\, \varphi(\xi) \simeq \mathrm i  {\partial \overline \partial}\, |\xi|^2$.

Recall that
$$
   \Gamma  =\left\{ \sum_{j=1}^N a_j K_{z_j}: N\in \mathbb{N}, a_j\in  \mathbb{C}, z_j\in \Cn  \textrm{ for } 1\le j\le N \right\}
$$
is dense in $F^p(\varphi)$ for all $0<p<\infty$.

\begin{lemma} \label{hankel-and-d-bar}
   Suppose  $1\le p \le \infty$.
\begin{itemize}
\item[(A)] There is a constant $C$ such that $\| A_\varphi(\omega)\|_{p, \varphi}\le C \left\|\omega\right \|_{p, \varphi}$ for $\omega\in L_{0, 1}$.

\item[(B)]
For $g\in \Gamma$ and  $f\in  C^2(\Cn)$ satisfying  $|\overline {\partial} f|\in L^p $, it holds that ${\overline \partial}  A_\varphi(g\overline \partial f)= g{\overline \partial}f$.
\end{itemize}
\end{lemma}

\begin{proof}
Let $z\in \C^n$. By \eqref{weights}, using Taylor expansion of $\varphi$ at $\xi$, we get
$$
	\varphi(z)-\varphi(\xi)\ge
	2 \mathrm{Re} \sum_{j=1}^n \fr {\partial \varphi(\xi)}{\partial \xi_j}(z_j-\xi_j) + m|z-\xi|^2.
$$
Then \eqref{A-operator} gives
\begin{multline}\label{d-bar-s-V1}
	\left| A_\varphi (\omega)(z) e^{-\varphi(z)}\right|\\
	\le C\int_{\Cn}\left |\omega (\xi)\right| e^{-\varphi(\xi)} \left\{\fr 1{|\xi-z|} +\fr 1{|\xi-z|^{2n-1}} \right\}e^{-m|\xi-z|^2} dv(\xi).
\end{multline}
For $l<2n$ fixed, define another integral operator $\mathcal A_l$ as
$$
	{\mathcal A_l} : h \mapsto \int_{\Cn} h(\xi) \fr {e^{-m|\xi-z|^2}} {|\xi-z|^ l} dv(\xi).
$$
It is easy to verify, by interpolation, that   $\mathcal A_l$ is bounded on $L^p$ for $1\le p\le \infty$. Therefore,
\begin{align*}
	\|A_\varphi (\omega)\|_{p, \varphi}
	&\le C \|\left (\mathcal A_1 + \mathcal A_{2n-1}\right) (|\omega|e^{-\varphi } )\|_{L^p}\\
	&\le C\left(\|{\mathcal A_1}\|_{L^p\to L^p} + {\mathcal A_{2n-1}}\|_{L^p\to L^p}\right)\left\|\omega\right \|_{p, \varphi},
\end{align*}
which completes the proof of part (A).

Notice that the convexity assumption in \eqref{weights} yields
$dd^c \varphi \simeq \omega_0$, which in turn means that $| {\partial \overline \partial} \varphi(\xi)|\simeq 1$. We use $p'$ to denote the conjugate of $p$, $\fr 1p +\fr 1{p'}=1$. Now, for $f\in C^2(\Cn)$ satisfying $|\overline {\partial} f|\in L^p $, and  $z, z_0\in \Cn$, we have
\begin{align*}
  &\int_{\Cn} |K(\xi, {z_0})\,  {\overline \partial}f(\xi)|\sum_{j=0}^{n-1}
      \fr {e^{-\varphi(\xi)} |{\overline \partial} \partial\varphi(\xi)|^j}{|\xi-z|^{2n-2j-1} }dv(\xi)\\
  &\quad\le C \Bigg\{\sup_{\xi\in B(z, 1)} |K(\xi, {z_0})\, {\overline
      \partial}f(\xi)e^{-\varphi(\xi)}|\\
      &\qquad\qquad+ \int_{\Cn\setminus B(z, 1)} |K(\xi, {z_0})\, {\overline \partial}f(\xi)|
      {e^{-\varphi(\xi)} }dv(\xi)\Bigg\}\\
   &\quad\le C e^{\varphi(z_0)}\left\{ \sup_{\xi\in B(z, 1)} | {\overline
      \partial}f(\xi)|  + \|{\overline \partial }f\|_{L^p}  \|K(\cdot, z_0) \|_{{p'}, \varphi} \right\}
	<\infty .
\end{align*}
Hence, for $g\in \Gamma$ and $z\in \Cn$, it holds that
$$
   \int_{\Cn} |g(\xi) {\overline \partial}f(\xi)|\sum_{j=0}^{n-1}
      \fr {e^{-\varphi(\xi)} |{\overline \partial} \partial\varphi(\xi)|^j}{ |\xi-z|^{2n-2j-1} }dv(\xi)<\infty.
$$
 From Proposition 10 of~\cite{BA82}, we get (B) (pay attention to the mistake in the last line of Proposition 10 in \cite{BA82} where $f$ is left out on the right hand side). The proof is completed.
\end{proof}

\begin{corollary}  \label{d-bar-hankel}
   Suppose $f\in  {\mathcal S}\cap C^1(\Cn)$ and  $|\overline {\partial} f|\in L^s $ with some $1\le s\le \infty$. For $g\in \Gamma$, it holds that
\begin{equation}\label{d-bar-s}
     H_f(g)= A_{\varphi} (g\overline \partial f)-P(A_{\varphi} (g\overline \partial f)).
\end{equation}
\end{corollary}

\begin{proof}
Given   $f\in  {\mathcal S}\cap C^1(\Cn)$  with  $|\overline {\partial} f|\in L^s $   and $g\in \Gamma$,  $\|g\overline \partial f\|_{1, \varphi} \le \|g \|_{s', \varphi} \| \overline \partial f\|_{L^s}<\infty$, where $s'$ is the conjugate of $s$. Lemma~\ref{hankel-and-d-bar} implies that $u=A_{\varphi} (g\overline \partial f)\in L^1(\varphi)$ and $\overline \partial u = g\overline \partial f$. Then $fg-u\in L^1(\varphi)$. Notice that
$
   {\overline \partial} \left( fg-u \right)=  g {\overline \partial}  f-  {\overline \partial}u=0
$, and so
$fg-u\in F^1(\varphi)$.
  Since $P|_{F^1_{\varphi}}={\mathrm{I}}$, we have
$$
 fg-u=  P\left(fg-u \right)= P(fg)- P(u).
$$
This shows that $H_f= u- P(u)$.
\end{proof}

\section{The space $\mathrm{IDA}$}\label{space-IDA}

In this section we introduce a new space to characterize boundedness and compactness of Hankel operators. The space $\IDA$ is related to the space of bounded mean oscillation $\BMO$ (see, e.g.~\cite{JN61, Zh12}), which has played an important role in many branches of analysis and their applications for decades. We find that $\IDA$ is also broad in scope and should have more applications in operator theory and related areas.

\subsection{Definitions and preliminary lemmas} Let  $0<q <\infty$ and $r>0$. For  $f\in L^q_{\mathrm{loc}}$ (the collection of  $q$-th locally Lebesgue integrable functions on $\Cn$), following Luecking in \cite{Lu92}, we define $G_{q,r}(f)$ as
\begin{equation}\label{G-q-r}
	G_{q,r}(f)(z)=\inf\left \{\left(\frac{1}{|B(z,r)|}\int_{B(z,r)}
	|f-h|^q dv \right)^{\frac{1}{q}}:h\in H(B(z, r))\right\}
\end{equation}
for $z\in \Cn$.

\begin{definition}\label{IDA def} Suppose $0< s\le \infty$ and $0<q< \infty$. The space $\mathrm{IDA}^{s, q}$ {\rm (}integral distance to holomorphic functions{\rm )} consists of all $f\in L^q_{\mathrm {loc}}$ such that
$$
	\|f\|_{\IDA^{s, q}} = \|G_{q, 1}(f)\|_{L^s}<\infty.
$$
The space $\IDA^{\infty, q} $ is also denoted by $\BDA^q$. The space  $\mathrm{VDA}^q$ consists of all $f\in \BDA^q$ such that
$$
	\lim_{z\to \infty} G_{q, 1}(f)(z)=0.
$$
\end{definition}

We will see in Section 6 that $\mathrm{IDA}^{s, q}$ is an extension of the space $\mathrm{IMO}^{s, q}$ introduced in \cite{HW18}.

Notice that the space $\BDA^2$ was first introduced in the context of the Bergman spaces of the unit disk by Luecking~\cite{Lu92}, who called it the space of functions with bounded distance to {\it analytic} functions (BDA).

\begin{remark}\label{proper-inclusions}
As is the case with the classical $\mathrm{BMO}^q$ and $\mathrm{VMO}^q$ spaces, we have
$$
	\BDA^{q_2} \subset \BDA^{q_1}\quad{\rm and}\quad \VDA^{q_2} \subset \VDA^{q_1}
$$
properly for $0<q_1<q_2<\infty$.
\end{remark}

Let $0<q<\infty$. For $z\in\C^n$, $f\in L^q( B(z, r), dv)$  and $r>0$, we define the $q$-th mean of $|f|$ over $B(z, r)$ by setting
$$
	M_{q, r}(f)(z)= \left(\frac{1}{|B(z,r)|}\int_{B(z,r)} |f|^q dv \right)^{\frac{1}{q}}.
$$
For $\omega \in L_{0, 1}$, we set $ M_{q, r}(\omega)(z)= M_{q, r}(|\omega|)(z)$.

\begin{lemma}\label{G-function}
 Suppose $0<q<\infty$. Then for  $f\in L^q _{\mathrm{loc}}$,  $z\in\C^n$ and $r>0$, there is some $h \in H(B(z, r))$ such that
\begin{equation}\label{h-z}
	M_{q, r}(f-h)(z)= G_{q, r}(f) (z)
\end{equation}
and
\begin{equation}\label{bounded-h}
	\sup_{w\in B(z, r/2)} |h (w)| \le C\|f\|_{L^q(B(z, r), dv)}
\end{equation}
where the constant $C$ is independent of $f$ and $r$.
\end{lemma}
\begin{proof}
Let $f\in L^q _{\mathrm{loc}}$, $z\in\C^n$ and $r>0$. Taking $h=0$ in the integrand   of (\ref{G-q-r}), we get
$$
	G_{q, r}(f)(z)\le M_{q, r}(f)(z)<\infty.
$$
Then for $ j=1, 2,\ldots$, we can pick $h_j\in H(B(z, r))$ such that
\begin{equation}\label{h-j-function}
	M_{q, r}(f-h_j)(z) \to G_{q, r}(f)(z)
\end{equation}
 as $j\to \infty$.
Hence, for $j$ sufficiently large,
\begin{equation}\label{h-z-q}
	M_{q, r}(h_j)(z)
	\le C \left  \{ M_{q, r}(f-h_j)(z)+ M_{q, r}(f)(z) \right\}
	\le C M_{q, r}(f)(z).
\end{equation}
This shows that $\{h_j\}_{j=1}^\infty$ is a normal family. Thus, we can find a subsequence $\{h_{j_k}\}_{k=1}^\infty$ and a function $h\in H(B(z, r))$ so that $\lim_{k\to \infty} h_{j_k}(w)\to h (w)$ for $w\in B(z, r)$.
By (\ref{h-j-function}),  applying Fatou's lemma, we have
$$
	G_{q, r}(f)(z)\le   M_{q, r}(f-h )(z) 
	\le \liminf_{k\to \infty}  M_{q, r}(f-h_{j_k})(z) = G_{q, r}(f)(z),
$$
which proves (\ref{h-z}). It remains to note that, with the plurisubharmonicity of $|h|^q$, for $w\in B(z, r/2)$, we have
$$
	|h (w)|\le M_{q, r/2}(h)(w)\le C M_{q, r }(h )(z)\le C M_{q, r}(f)(z),
$$
which completes the proof.
\end{proof}

\smallskip

\begin{corollary}\label{G-comparison}
For $0<s<r$, there is a constant $C>0$ such that for $f\in L^q_{loc}$ and $w\in B(z, r-s)$, it holds that
\begin{equation}\label{G-comp-a}
	G_{q, s}(f)(w)\le M_{q, s}(f-h)(w)\le C G_{q, r}(f)(z),
\end{equation}
where $h$ is as in Lemma \ref{G-function}.
\end{corollary}

\begin{proof}
For $0<s<r$ and  $w\in B(z, r-s)$, $B(w, s)\subset B(z, r)$. Then, the first estimate in (\ref{G-comp-a}) comes from the definition of $G_{q, s}(f)$, while \eqref{h-z} yields
$$
	M_{q, s}(f-h)(w)\le C M_{q, r}(f-h)(z) =C G_{q, r}(f)(z)
$$
which completes the proof.
\end{proof}

For $z\in \Cn$ and $r>0$, let
$$
      A^q(B(z, r), dv)=  L^q(B(z, r), dv)\cap H(B(z, r))
$$
be the $q$-th Bergman space over $B(z, r)$. Denote by $P_{z,r}$ the corresponding Bergman  projection induced by the Bergman kernel for $A^2(B(z, r), dv)$. It is well know that $P_{z, r}(f)$ is well defined for $f\in L^1\left((B(z, r), dv\right)$.

\begin{lemma}\label{Project}
Suppose $1\le q <\infty$ and $0<s<r$. There is a constant $C>0$ such that, for $f\in L^q _{\mathrm{loc}}$ and $w\in B(z, \frac {r-s}2)$,
\begin{equation}\label{P-z-r}
	G_{q, s}(f) (w) 
	\le M_{q, s}\left(f-P_{z, r}(f)\right)(w) 
	\le C G_{q, r}(f) (z) \,\, \textrm{ for }\,\, z\in \Cn.
\end{equation}
\end{lemma}

\begin{proof}
We only need to prove the second inequality. Suppose  $1< q<\infty$. Notice that $P_{0, 1}$ is the standard Bergman  projection on the unit ball of $\Cn$. Theorem 2.11 of~\cite{Zh05} implies that
$$
	\|P_{0, 1}\|_{L^q(B(0, 1), dv) \to A^q(B(0, 1), dv)}<\infty.
$$
Now for   $r>0$ fixed and $f\in L^q((B(0, r), dv)$,  set $f_r(w)= f(rw)$. Then
$$
	\|f_r\|_{L^q(B(0, 1), dv)} = r^{-2n/q} \|f\|_{L^q(B(0, 1), dv)}.
$$
Furthermore, it is easy to verify that the operator   $f \mapsto P_{0, 1}(f_r)\left(\fr \cdot r\right)$ is self-adjoint and idempotent, and it maps $L^2((B(0, r), dv)$ onto $A^2((B(0, r), dv)$. Therefore,
   $$
       P_{0, r}(f)(z) = P_{0, 1}(f_r)\left(\fr zr\right) \ \textrm{ for } \ f\in L^q(B(0, r), dv),
$$
and hence
$$
    \|P_{0, r}\|_{L^q(B(0, r), dv) \to A^q(B(0, r), dv)} =\|P_{0, 1}\|_{L^q(B(0, 1), dv) \to A^q(B(0, 1), dv)}.
$$
Now for $z\in \Cn$ and $r>0$, using a suitable dilation, it follows that
 \begin{equation}\label{B-projection-1}
    \|P_{z, r}\|_{L^q(B(z, r), dv) \to A^q(B(z, r), dv)}=\|P_{0, 1}\|_{L^q(B(0, 1), dv) \to A^q(B(0, 1), dv)}<\infty.
 \end{equation}
Unfortunately,  $P_{z,r}$ is not  bounded on $L^1(B(z, r), dv)$, but with the same approach as above, by  Theorem 1.12 of~\cite{Zh05} and Fubini's theorem, we have
\begin{equation}\label{B-projection-2}
	\|P_{z,r}\|_{L^1(B(z, r), dv)\to  A^1(B(z, r), (r^2-|\cdot\, -z|^2)dv)}\le C
\end{equation}
for $z\in \Cn$ and $r>0$.

Choose $h$ as in Lemma \ref{G-function}. Then $h\in A^q(B(z, r), dv)$ because $f\in L^q_{\loc}$.   Thus, $P_{z,r}(h)=h$. Now for $w\in B(z, (r-s)/2)$ and $1\le q<\infty$,
\begin{equation}\label{Projet-in-Br}
\begin{split}
	&\left\{  \int_{B(w, s)}\left|f-P_{z,r}(f) \right|^q dv\right\}^{\frac1q}\\
	&\le  C \left\{  \int_{B(z,(r+s)/2)} \left|f-P_{z,r}(f) \right|^q dv\right\}^{\fr 1 q}\\
	&\le  C  \left\{  \int_{B(z, r)} \left|f(\xi)-P_{z,r}(f)(\xi) \right|^q (r^2-|\xi-z|^2) dv(\xi)\right\}^{\frac1q}\\
	&\le C\Bigg\{   \left[ \int_{B(z, r  )} \left|f-h \right|^q  dv\right]^{\fr 1 q}\\
	&\qquad\qquad+\left[  \int_{B(z, r  )} \left|P_{z,r} (f-h)(\xi) \right|^q (r^2-|\xi-z|^2) dv(\xi) \right]^{\fr 1 q}\Bigg\}\\
	&\le  C \left\{ \int_{B(z, r  )} \left|f-h \right|^q  dv\right\}^{\fr 1 q}.
\end{split}
\end{equation}
From this and Lemma \ref{G-function}, \eqref{P-z-r} follows.
\end{proof}

Given $t>0$, let $\{a_j\}_{j=1}^\infty$ be a $\fr t2$-lattice,  set $J_z  = \{j: z\in B(a_j, t)\}$ and denote by  $|J_z|$ the cardinal number of $J_z$.  By (\ref{lattice}), $|J_z|= \sum_{j=1}^\infty \chi_{B(a_j, t)}(z)\le N$. Choose a partition of unity $\{\psi_j\}_{j=1}^\infty$ subordinate to $\{B(a_j, t/2)\}$ such that
\begin{equation}\label{psi-1}
\begin{split}
	&\supp \psi_j\subset B(a_j, t/2),  \, \, \psi_j(z)\ge 0,  \,\,\sum_{j=1}^\infty \psi_j(z)=1,\\
	&|\overline \partial \psi_j(z)|\le C t^{-1}, \quad \sum_{j=1}^\infty \overline \partial \psi_j(z) =0.
\end{split}
\end{equation}
Given $f\in L^q_{\mathrm{loc}}$, for $j=1,2,\ldots$, pick $h_j\in H(B(a_j, t))$ as in Lemma \ref{G-function} so that
$$
	M_{q, t}(f-h_j)(a_j) = G_{q, t} (f)(a_j).
$$
Define
\begin{equation}\label{decomposition-a}
	f_1=\sum_{j=1}^\infty h_j \psi_j\ \textrm{ and }\  f_2= f-f_1.
\end{equation}
Notice that $f_1(z)$ is a finite sum for every $z\in \C^n$ and hence well defined because we have $\supp \psi_j\subset B(a_j,t/2)\subset B(a_j,t)$.

Inspired by a similar treatment on pages 254--255 of~\cite{Lu92}, using the partition of unity, we can prove the following estimate.

\begin{lemma}\label{basic-docomp}
Suppose $0< q<\infty$. For $f\in L^q_{\mathrm{loc}}$ and $t>0$, decomposing $f=f_1+f_2$ as in \eqref{decomposition-a}, we have $f_1\in C^2(\Cn)$  and
\begin{equation}\label{decomposition-b}
 \left|\overline{\partial}f_1(z)\right |+     M_{q, t/2}(\overline \partial f_1)(z)  +   M_{q, t/2}(f_2 )(z)\le C
   G_{q,2t}(f)(z)
\end{equation}
for $ z\in \Cn$, where the constant $C$ is independent of $f$.
\end{lemma}
\begin{proof} Observe first that $f_1\in C^2(\C^n)$ follows directly from the properties of the functions $h_j$ and $\psi_j$.  For $z\in \Cn$,  we may assume $z\in B(a_1, t/2)$ without loss of generality. Then for those $j$ that satisfy $\overline \partial \psi_j(z)\neq 0$, $|h_j-h_1|^q$ is plurisubharmonic on $B(z, t/2)\subset B(a_j, t)$.  Hence, by Corollary \ref{G-comparison},
\begin{align*}
	\left|\overline{\partial}f_1(z)\right| 
	&=\left|\sum_{j=1}^\infty  \left( h_j(z)-h_1(z)\right) \overline \partial \psi_j(z)\right|\\
	&\le \sum_{j=1}^\infty  \left|h_j(z)-h_1(w)\right|  |\overline \partial \psi_j(z)|\\
	&\le C \sum_{\{j: \, |a_j-z|< t/2\}} M_{q, t/4}(h_j -h_1)(z)\\
	&\le C  \sum_{\{j: \, |a_j-z|< t/2\}} \left[ M_{q, t/4}(f-h_j)(z) + M_{q, t/4}(f-h_1)(w)\right]\\
	&\le C \sum_{\{j: \, |a_j-z|< t/2 \}} G_{q, t}(f)(a_j).
\end{align*}
Thus, using Corollary \ref{G-comparison} again, we get
$$
	\left|\overline{\partial}f_1(z)\right |
	\le C G_{q, 3t/2}(f)(z) \ \textrm{ for } \ z\in \C^n,
$$
and so,
\begin{align*}
	M_{q, t/2}(\overline{\partial}f_1)(z)^q
	&\le C \fr 1 {|B(z, t/2)|} \int_{B(z, t/2)}G_{q, 3t/2}(f)(w)^q dw\\
	&\le C G_{q, 2t}(f)(z)^q.
\end{align*}
Similarly, we have $ |f_2(\xi)|^q\le C \sum_{j=1}^\infty | f(\xi)-h_j(\xi)|^q \psi_j(\xi)^q$, and so
\begin{align*}
	M_{q, t/2}(f_2)(z)^q &\le C\sum_{j=1}^\infty \frac1{|B(z,t/2)|}\int_{B(z, t/2)} |f-h_j|^q \psi_j^q \,dv \\
	&\le C \sum_{\{j:\, |a_j-z|< t/2\}} G_{q, t}(f)(a_j)^q.
\end{align*}
Therefore,
$$
         M_{q, t/2}(f_2)(z)  \le C G_{q, 3t/2}(f)(z).
$$
Combining this and the other two estimates above gives~\eqref{decomposition-b}.
\end{proof}

Given $\{\psi_j\}$ as in \eqref{psi-1}, we have another decomposition $f= \mathfrak{F}_1+\mathfrak{F}_2$, where
\begin{equation}\label{P-decompositon}
	\mathfrak{F}_1=\sum_{j=1}^\infty P_{a_j, t}(f)  \psi_j
	\ \textrm{and}\  \mathfrak{F}_2= f-\mathfrak{F}_1.
\end{equation}
When $q=2$, the two decompositions coincide.

\begin{corollary} \label{proj-decomp}
Suppose $1\le  q<\infty$. For $f\in L^q_{\mathrm{loc}}$ and $t>0$,  we have $\mathfrak{F}_1\in C^2(\Cn)$  and
\begin{equation}\label{pro-decomp-a}
	\left|\overline{\partial}\mathfrak{F}_1(z)\right |
	+ M_{q, t/2}(\overline \partial \mathfrak{F}_1)(z)
	+ M_{q, t/2}(\mathfrak{F}_2 )(z)\le C G_{q,2t}(f)(z)
\end{equation}
for $ z\in \Cn$, where the constant $C$ is independent of $f$.
\end{corollary}

\begin{proof}
The proof can be carried out as that of Lemma \ref{basic-docomp} using~\eqref{P-z-r} instead of~\eqref{G-comp-a}. We omit the details.
\end{proof}

\subsection{The decomposition}\label{decomposition subsection} In our analysis, we will appeal to $\overline \partial$-techniques several times. Let $\Omega\subset \Cn$ be strongly pseudoconvex with $C^4$ boundary, and let $S$ be a $\overline \partial $-closed $(0, 1)$ form on $\Omega$ with $L^p$ coefficients, $1\le p\le \infty$. As in~\cite{Kr82}, we denote by ${\mathrm H}_{\Omega}( S) $  the   Henkin's solution of $\overline \partial$-equation  $\overline \partial   u =S$ on $\Omega$.  We observe that Theorem 10.3.9 of~\cite{Kr82} implies that, for $1\le q<\infty$,
\begin{equation}\label{Henkin-solution}
      \|{\mathrm H}_{\Omega}( S)\|_{L^q(\Omega, dv)}\le C  \|  S\|_{L^q(\Omega, dv)},
\end{equation}
where the constant $C$ is independent of $S$ and of ``small'' pertubations of the boundary. (We note that the second item in Theorem~10.3.9 of~\cite{Kr82} is stated incorrectly and should read $\|u\|_{L^q} \leq C_p \|f\|_p$ instead.) Indeed, to deduce \eqref{Henkin-solution}, we consider three cases. First, for $1\le q<\frac {2n+2}{2n+1}$,
$$
    \|{\mathrm H}_{\Omega}( S)\|_{L^q(\Omega, dv)}\le C \|S\|_{L^1(\Omega, dv)}\le C  \|  S\|_{L^q(\Omega, dv)}.
$$
For $ q=\frac {2n+2}{2n+1}$, take $1< p=q<2n+2$ and $q_1=\frac {2n+2}{2n}>q$. Then $\frac 1{q_1}=\frac 1 p -\frac 1{2n+2}$, and by the second item in Theorem 10.3.9 of~\cite{Kr82}, we have
$$
    \|{\mathrm H}_{\Omega}( S)\|_{L^q(\Omega, dv)}\le C
     \|{\mathrm H}_{\Omega}( S)\|_{L^{q_1}(\Omega, dv)}
    \le C  \|  S\|_{L^p(\Omega, dv)}.
$$
Finally, for $ q> \frac {2n+2}{2n+1}$, choose $p$ so that $\frac 1q=\frac 1 p -\frac 1{2n+2}$. Then $1<p<2n+2$ and $p<q$. Now Theorem 10.3.9 of~\cite{Kr82} implies
   $$
    \|{\mathrm H}_{\Omega}( S)\|_{L^q(\Omega, dv)}
    \le C  \|  S\|_{L^p(\Omega, dv)} \le C  \|  S\|_{L^q(\Omega, dv)}.
$$

\smallskip

\begin{theorem}\label{BDA-equivalence}
Suppose $1\le q<\infty$, $0<s<\infty$, and  $f\in L^q_{\mathrm {loc}}$. Then $f\in  \mathrm{IDA}^{s, q}$ if and only if $f$ admits a decomposition $f=f_1+f_2$  such that
\begin{eqnarray}\label{q-less-then-p-m}
	f_1\in C^2(\Cn), \quad M_{q, r }(\overline \partial f_1)+ M_{q, r}(f_2) \in L^{s}
\end{eqnarray}
for some (or any) $r>0$.  Furthermore, for fixed $\tau, r>0$, it holds that
\begin{equation}\label{norm-euival}
	\|f\|_{\IDA^{s, q}} \simeq \left\| G_{q, \tau}(f)\right\|_{L^s}
	\simeq  \inf \left \{ \|M_{q, r }(\overline \partial f_1) \|_{L^{s}} + \|M_{q, r}(f_2)\|_{L^{s}}\right\}
\end{equation}
where the infimum is taken over all possible decompositions $f=f_1+f_2$ that satisfy \eqref{q-less-then-p-m} with a fixed $r$.
\end{theorem}

\begin{proof}
First, given $0<r<R<\infty$, we have some $\mathbf a_1, \mathbf a_1, \cdots, \mathbf a_m\in \Cn$ so that $B(0, R)\subset \cup_{j=1}^m B(\mathbf a_j, r)$. Then, for $g\in L^q_{\mathrm{loc}}$,
$$
     M_{q, R}(g )(z)^s\le C \sum_{j=1}^m  M_{q,  r}(g)(z+ \mathbf a_j)^s, \,\, z\in \Cn,
$$
and
\begin{equation}\label{norm-equ}
\begin{split}
	\int_{\Cn}  M_{q, R}(g )(z)^sdv(z) &\le C \sum_{j=1}^m \int_{\Cn} M_{q,  r}(g)(z+ \mathbf a_j)^sdv(z)\\
	&\le C \int_{\Cn} M_{q,  r}(g)(z)^sdv(z).
\end{split}
\end{equation}
This implies that \eqref{q-less-then-p-m} holds for some $r$ if and only if it holds for any $r$.

Suppose that $f\in L^q_{\mathrm{loc}}$ with $\|G_{q, \tau}(f)\|_{L^s}<\infty$ for some $\tau>0$ and decompose $f=f_1+f_2$ as in Lemma \ref{basic-docomp} with  $t=\fr {\tau}2$.  Then $f_1\in C^2(\Cn)$ and
$$
	\left|\overline{\partial}f_1(z)\right | +
	M_{q, \tau/4}(\overline \partial f_1)(z) + 
	M_{q, \tau/4}(f_2 )(z)\le C G_{q,\tau}(f)(z).
$$
Now for any $r>0$,  we have
\begin{equation}\label {M-est-1}
	\| M_{q, r}( \overline \partial  f_1)\|_{L^s} + \| M_{q, r}(   f_2)\|_{L^s}
	\le   C  \| G_{q, \tau}(f)\|_{L^s}.
\end{equation}
This implies that, $f=f_1+f_2$ satisfies   (\ref{q-less-then-p-m}).

Conversely, suppose $f=f_1+f_2$ with $f_1\in C^2(\Cn)$ and $M_{q, r}(\overline \partial f_1)+ M_{q, r}(f_2) \in L^{s}$ for some $r>0$ as in Theorem~\ref{BDA-equivalence}. Then for any  $\tau>0$,
\begin{equation}\label{f-2-norm}
	\|G_{q, \tau}(f_2)\|_{L^s} \le C   \|M_{q, \tau}(f_2)\|_{L^s} 
	\le C \|M_{q, r}(f_2)\|_{L^s} .
\end{equation}
So $f_2\in \mathrm{IDA}^{s, q}$.
To consider $f_1$, we write $ u= {\mathrm H}_{B(z, 2\tau)}(\overline \partial f_{1} ) $  for the Henkin  solution of the equation  $\overline \partial u= \overline \partial f_1$ on $B(z, 2\tau)$.  From  (\ref{Henkin-solution}) and (\ref{q-less-then-p-m}), $u$ satisfies
\begin{equation}\label{d-bar-M}
      M_{q, 2\tau}(u)(z)\le C M_{q, 2\tau}(\overline \partial f_1)(z) \, \, \textrm{ for }\,\, z\in \C^n,
\end{equation}
which implies that $u\in L^q(B(z, 2\tau), dv)$. Similarly to (\ref{Projet-in-Br}),
$$
	M_{q, \tau}\left({\mathrm P} _{z, 2\tau} ( u)\right)(z)
	\le C  M_{q, 2\tau}(u)(z).
$$
Thus,
\begin{equation}\label{d-bar-M-1}
\begin{split}
	M_{q, \tau}\left(u - {\mathrm P} _{z, 2\tau} ( u )\right )(z)
	&\le    M_{q, \tau}\left(u\right )(z) + M_{q, \tau}\left({\mathrm P} _{z, 2\tau} (u ) \right )(z)\\
	&\le    C M_{q, 2\tau}( u  ) (z).
\end{split}
\end{equation}
Since
$$
	f_1- u \in L^q(B(z, 2\tau), dv) \, \,\textrm{and}\,\, 
	\overline \partial \left(f_1 -u\right)=0,
$$
we have
$$
	f_1- u  \in A^q(B(z, 2\tau), dv).
$$
Notice that $P_{z, 2\tau}|_{A^q(B(z, 2\tau), dv)}= \mathrm{I}$, and so
\begin{equation}\label{d-bar-M-2}
	f_{ 1}(\xi) - {\mathrm P} _{z, 2\tau} ( f_{ 1})(\xi)
	= u(\xi) - {\mathrm P} _{z, 2\tau}\left(u\right )(\xi)   \,\, 
	\textrm{ for }\, \xi \in B(z, 2\tau).
\end{equation}
Combining \eqref{d-bar-M}, \eqref {d-bar-M-1} and \eqref {d-bar-M-2}, we get
\begin{align*}
     M_{q,\tau}\left( f_{ 1} - {\mathrm P}_{z, 2\tau} ( f_{ 1})\right)(z)
     &= M_{q, \tau}\left(u - {\mathrm P}_{z, 2\tau}\left(u\right )\right)(z)\\
     &\le M_{q, 2\tau}\left(u \right)(z) \le C M_{q, 2\tau}\left( \overline \partial f_1  \right)(z).
\end{align*}
Therefore, by (\ref{norm-equ}),
\begin{align*}
	\|G_{q, \tau} ( f_{ 1})\|_{L^s}
	&\le  \|M_{q, r}\left( f_{ 1} - {\mathrm P} _{z, 2\tau} ( f_{ 1})\right)\|_{L^s}\\
	&\le C \|  M_{q, 2\tau}\left( \overline \partial f_1  \right)  \|_{L^s} \le C  \|  M_{q, r}\left( \overline \partial f_1  \right)  \|_{L^s}.
\end{align*}
This and \eqref{f-2-norm} yield
\begin{equation}\label{f-decomposition}
	\|G_{q, \tau} ( f)\|_{L^s}  
	\le C \left\{ \|  M_{q, r}\left( \overline \partial f_1  \right)  \|_{L^s}+  \|M_{q,r}(f_2)\|_{L^s} \right\}.
\end{equation}
Thus, $f=f_1+f_2 \in \mathrm{IDA}^{s, q}$.

It remains to note that the norm equivalence \eqref{norm-euival} follows from \eqref{M-est-1} and \eqref{f-decomposition}.
\end{proof}

With a similar proof we have the following corollary.

\begin{corollary}\label{VDA-1}
Suppose $1\le q<\infty$, and $f\in L^q_{\mathrm {loc}}$. Then $f\in  \mathrm{BDA}^q$ {\rm (}or $ \mathrm{VDA}^q${\rm )} if and only if  $f=f_1+f_2$, where
\begin{eqnarray}\label{VDA}
	f_1\in C^2(\Cn), \, \, \overline{\partial}f_1 \in L^\infty_{0, 1} \quad (\textrm{or}   \lim_{z\to \infty} |\overline{\partial}f_1|=0 )
\end{eqnarray}
and
\begin{equation} \label{VDA-a}
	M_{q, r}(f_2) \in L^\infty \quad (\textrm{or}  \lim_{z\to \infty} M_{q, r}(f_2)=0)
\end{equation}
for some (or any) $r>0$. Furthermore,
$$
	\|f\|_{\BDA^q} \simeq
	\inf\{ \|\overline \partial f_1\|_{L^\infty_{0, 1}}+ \|M_{q, r}(f_2)\|_{L^\infty}\},
$$
where the infimum is taken over all possible decompositions $f=f_1+f_2$ with $f_1$ and $f_2$  satisfying the conditions in~\eqref{VDA} and~\eqref{VDA-a}.
\end{corollary}

\begin{corollary}\label{r independence} Suppose $1\le q<\infty$. Different values of $r$ give equivalent  seminorms  $\|G_{q, r} ( \cdot)\|_{L^s}$ on $\mathrm{IDA}^{s, q}$ when $0<s<\infty$ and on both $\mathrm{BDA}^q$ and $\mathrm{VDA}^q$ when $s=\infty$.
\end{corollary}

\begin{remark}\label{IDA-BMO}
Recall that each $f$ in $\mathrm{BMO}^{q}$ can be decomposed as $f = f_1 + f_2$, where $f_1$ is of bounded oscillation $\mathrm{BO}$ and $f_2$ has a bounded average $\mathrm{BA^q}$ (see~\cite{Zh12} for the one-dimensional case and~\cite{Lv19} for the general case). Furthermore, we may choose $f_1$ to be a Lipschitz function in $C^2(\C^n)$ (see Corollary 3.37 of \cite{Zh12}), that is, $f\in \BMO^{q}$ if and only if $f=f_1+f_2$ with all $\frac {\partial f_1}{\partial x_j}\in L^\infty$ for $j=1, 2, \ldots, 2n$ and $f_2\in \mathrm{BA}^q$, or in the language of complex analysis both $\overline{\partial}  f_1$ and $\overline{\partial} \bar {f_1}$ are bounded. Therefore, $f\in \mathrm{BMO}^{q}$ if and only if $f, \bar f \in \mathrm{BDA}^q$. For a similar relationship between $\mathrm{IMO}^q$ and the $\IDA$ spaces, see Lemma 6.1 of \cite{HV22} and Theorem \ref{doubling-bdd} below.
\end{remark}

\subsection{$\IDA$ as a Banach space} We next prove that $\IDA^{s,q}/H(\C^n)$ with $1\le s, q< \infty$ is a Banach space when equipped with the induced norm
\begin{equation}\label{e:quotient norm}
	\| f + H(\C^n)\| = \|f\|_{\IDA^{s,q}}
\end{equation}
for $f\in \IDA^{s,q}$. 

\begin{theorem}\label{complete}
For $1\le s, q< \infty$, the quotient  space   $\mathrm{IDA}^{s, q}/H(\Cn)$ is a Banach space with the norm induced by $\|\cdot\|_{\IDA^{s,q}}$.
\end{theorem}

\begin{proof}
Obviously $H(\C^n)\subset \IDA^{s,q}$. Now given $f\in \mathrm{IDA}^{s, q}$ and $h\in H(\Cn)$,    $G_{q,  r}(f) =  G_{q, r}(f+h)$.  This means that the norm in~\eqref{e:quotient norm} is well defined on $\mathrm{IDA}^{s, q}/H(\Cn)$.  If $\|f\|_{\IDA^{s, q}}=0$,  then $G_{q, r}(f)(z)= 0$ in $\Cn$. By Lemma \ref{G-function}, $f\in H(B(z, r))$ and hence $f\in H(\Cn)$.

Let $f_1, f_2\in \mathrm{IDA}^{s, q}$ and $z\in\C^n$. According to Lemma \ref{G-function}, there are functions $h_j$ holomorphic in $B(z,r)$ such that
$$
	M_{q, r}(f_j-h_j)(z)= G_{q, r}(f_j)(z) \,\, \textrm{ for }\, j=1, 2.
$$
Then, since
$$
	M_{q, r}\left((f_1+f_2)-(h_1+h_2)\right)(z) 
	\le M_{q, r}\left(f_1 - h_1\right)(z) +M_{q, r}\left(f_2 - h_2\right)(z),
$$
we have
$$
      G_{q, r}(f_1+f_2)(z)\le G_{q, r}(f_1)(z)+ G_{q, r}(f_2)(z)\,\, \textrm{ for }\,\, z\in \C^n.
$$
Hence, $\|f_1+f_2\|_{\IDA^{s, q}}\le \|f_1 \|_{\IDA^{s, q}}+ \|f_2\|_{\IDA^{s, q} }$.  In addition, $\|f\|_{\IDA^{s, q}}\ge 0$ and $\|af\|_{\IDA^{s, q}}=|a|\|f\|_{\IDA^{s, q}}$ for $a\in \mathbb C$. Therefore, $\|\cdot\|_{\IDA^{s, q}}$ induces a norm on $\mathrm{IDA}^{s, q}/H(\C^n)$.

It remains to prove that the norm is complete. Suppose that $\{f_m\}_{m =1}^\infty$ is a Cauchy sequence in
$$
     \|\cdot\|_{\IDA^{s, q}} =  \|  G_{q, 1}(\cdot)\|_{L^s}.
$$
According to Corollary \ref{r independence}, we may assume that $\{f_m\}_{m =1}^\infty$ is a Cauchy sequence in $ \|  G_{q, r}(\cdot)\|_{L^s}$ with $r>0$ fixed. We now embark on proving that, for some  $f\in \mathrm{IDA}^{s, q}$, $\lim_{m\to \infty} \| G_{q, r/2}(f_m-f)\|_{L^s}=0$, which implies $\{f_m\}_{m=1}^\infty$ converges  to some $f\in \mathrm{IDA}^{s, q}$ in $\|\cdot\|_{\IDA^{s,q}}$-topology.  For this purpose, let $\{a_j\}_{j=1}^\infty$ be some $t=r/4$-lattice.  We decompose each $f_m$ similarly to \eqref{P-decompositon} as follows
$$
	f_{m, 1}= \sum_{j=1}^\infty P_{a_j, r}(f_m) \psi_j\quad {\rm and}\quad f_{m,2} = f_m-f_{m,1},
$$
where $\{\psi_j\}_{j=1}^\infty $ is the partition of unity subordinate to  $\{B(a_j, r/4)\}_{j=1}^\infty$ as in (\ref{psi-1}). It follows from Corollary~\ref{proj-decomp} that
\begin{align*}
	M_{q, r/8}(f_{m, 2}-f_{k, 2})(z)^s
	&= M_{q, r/8}\left((f_m-f_k) - \sum_{j=1}^\infty  P_{a_j, t}(f_m-f_k))\psi_j\right)(z)^s\\
	&\le C G_{q, r/2}(f_{m}-f_{k})(z)^s\\
	&\le C \int_{B(z, r/2)} G_{q, r }(f_{m}-f_{k})(\xi)^s dv(\xi).
\end{align*}
This implies that $\{f_{m,2}\}_{j=1}^\infty$ converges to some function $f_2$ in $L^q_{\mathrm{loc}}$-topology. In addition, by Lemma~\ref{Project}, we have
$$
	M_{q, r/2}\left(f_{m, 2}-f_{k, 2}- \mathrm P_{z, r}(f_{m, 2}-f_{k, 2})\right)(z) \le C G_{q,r}(f_{m, 2}-f_{k, 2})(z).
$$
Letting $k\to \infty$ and applying Fatou's Lemma, we get
\begin{align*}
	G_{q, r/2}(f_{m, 2}-f_{ 2})(z)^s
	&\le  M_{q, r/2}\left(f_{m, 2}-f_{  2}- \mathrm P_{z, r}(f_{m, 2}-f_{  2})\right)(z)^s\\
	&\le C \liminf_{k\to \infty}  G_{q,r}(f_{m, 2}-f_{k, 2})(z)^s.
\end{align*}
Integrate both sides over $\C^n$ and apply Fatou's lemma again to obtain the estimate
$$
	\int_{\Cn} G_{q, r/2}(f_{m, 2}-f_{2})^s \mathrm d v
	\le C   \liminf_{k\to \infty} \|f_{m, 2}- f_{k, 2}\|_{\mathrm{IDA}^{s, q}}.
$$
Therefore,
\begin{equation}\label{f-m-2}
	\lim_{m\to \infty} \| f_{m, 2}-f_2\|_{\IDA^{s, q}} =0.
\end{equation}

Next we consider $\{f_
{m, 1}\}_{m=1}^\infty$. Applying the estimate \eqref{pro-decomp-a} to $f_m-f_k$,
\begin{equation}\label{d-bar-est}
	\left |\overline \partial \left( f_{m, 1}- f_{k, 1} \right)(z)\right|
	\le C G_{q, r/2}\left( f_{m}- f_{k} \right)(z).
\end{equation}
Hence,  $  \{\overline \partial  f_{m, 1}\}_{m=1}^\infty$ is a Cauchy sequence in $ L^s_{0, 1}$ (see~\eqref{norm of differential form}). We may assume $\overline \partial  f_{m, 1}\to S=\sum_{j=1}^n S_j \mathrm{d} \overline z_j$ under the $ L^s_{0, 1}$-norm. Since $\overline \partial ^2=0$, $\overline \partial  f_{m, 1}$ is trivially $\overline \partial $-closed, and so, as the $ L^s_{0, 1}$ limit of $\{\overline \partial  f_{m, 1}\}_{m=1}^\infty$, $S$ is also $\overline \partial $-closed weakly. Let $\phi(z)= \fr 12 |z|^2$  and $g=1\in \Gamma$, and define
$$
	f_1(z) =  A_{\phi}(S), \, \textrm{ and } \, f_{m, 1}^*
	= A_{\phi}(\overline \partial f_{m, 1}).
$$
Then, by Lemma \ref{hankel-and-d-bar},
$$
	f_1, \, \, f_{m, 1}^* \in L^s\left(\phi\right )\subset L^s_{\mathrm{loc}}, \, \, 
	\overline \partial  f_{m, 1}^* = \overline \partial  f_{m, 1},
$$
and $\{f_{m, 1}^*\}_{m=1}^\infty$ converges to $f_1$ in $L^s\left(\phi\right)$. Therefore, for $\psi \in C^\infty_c(\Cn)$ (the family of all $C^\infty$ functions with compact support) and $j=1, 2, \cdots, n$, it holds that
\begin{align*}
	- \left\langle  f_1,  \fr { \partial \psi}{  \partial  z_j}  \right\rangle_{L^2} 
	&= -\lim_{m\to \infty} \left\langle  f_{m, 1}^*,  \fr { \partial \psi}{\partial z_j} \right\rangle_{L^2}
	=\lim_{m\to \infty} \left\langle  \fr { \partial f_{m, 1}^*}{ \partial  \overline z_j}, \psi \right\rangle_{L^2}\\
	&=\lim_{m\to \infty} \left\langle  \fr { \partial f_{m, 1}}{ \partial  \overline z_j} ,  \psi \right\rangle_{L^2}= \left\langle S_j ,  \psi \right\rangle_{L^2}.
\end{align*}
Hence, $\overline \partial f_1 = S$ weakly. Then for  ${\mathrm H}_{B(z, r)}(\overline \partial f_{m, 1}- S)$, the Henkin solution to the equation $\overline \partial u =\overline \partial f_{m, 1} - S$ on   $B(z, r)$, (\ref{Henkin-solution}) gives
\begin{equation}\label{Henkin-a}
	\|{\mathrm H}_{B(z, r)}(\overline \partial f_{m, 1} - S)\|_{L^q(B(z, r), {\mathrm d}v)}
	\le C \|\overline \partial f_{m, 1}- S\|_{L^q(B(z, r), {\mathrm d}v)}.
\end{equation}
In addition, according to~\eqref{d-bar-M-2}, it holds that
$$
	(f_{m, 1}- f_{ 1}) - {\mathrm P} _{z, r} (f_{m, 1}- f_{ 1})
	=  {\mathrm H}_{B(z, r)}(\overline \partial f_{m, 1}- S) -  {\mathrm P} _{z, r}\left( {\mathrm H}_{B(z, r)}(\overline \partial f_{m, 1}- S)\right)
$$
on $B(z, r)$. Therefore, by  (\ref{B-projection-1}), (\ref{B-projection-2}), and (\ref {Henkin-a})  we have
\begin{equation}\label{d-bar-m}
\begin{split}
	& \|(f_{m, 1}- f_{ 1}) - {\mathrm P} _{z, r} (f_{m, 1}- f_{ 1})\|_{L^q(B(z, r/2),\, dv)}^q\\
	&=  \| {\mathrm H}_{B(z, r)}(\overline \partial f_{m, 1}- S)-  {\mathrm P} _{z, r}\left( {\mathrm H}_{B(z, r)}(\overline \partial f_{m, 1}- S)\right )\|_{L^q(B(z, r/2),\, dv)}^q\\
	&\le  C   \| {\mathrm H}_{B(z, r)}(\overline \partial f_{m, 1}- S)
   \|_{L^q(B(z, r),\, dv)}^q\\
   &\le  C  \|\overline \partial f_{m, 1}- S\|_{L^q(B(z, r), {\mathrm d}v)}^{q}.
\end{split}
\end{equation}
Since $S=\lim_{k\to \infty}   \overline \partial f_{k, 1}$ in $L^s_{0, 1}$, by Fatou's lemma,
\begin{equation}\label{d-bar-z}
\begin{split}  
	\|\overline \partial f_{m, 1}- S\|_{L^q(B(z, r), {\mathrm d}v)}^{q}
	&\le C \liminf_{k\to \infty} \| \overline \partial \left( f_{m, 1} -f_{k, 1}\right)  \|_{L^q(B(z, r), dv)}^q\\
	&\le  C \liminf_{k\to \infty} G_{q, 2r}( f_{m, 1} -f_{k, 1})(z)^{q},
\end{split}
\end{equation}
where the last inequality follows from \eqref{d-bar-est}. We combine \eqref{d-bar-m} and \eqref{d-bar-z} to get
\begin{multline*}
	\|(f_{m, 1}- f_{ 1}) - {\mathrm P} _{z, r} (f_{m, 1}- f_{ 1})\|_{L^q(B(z, r/2), dv)}^s\\
	\le  C \liminf_{k\to \infty} G_{q, 2r}( f_{m, 1} -f_{k, 1})(z)^s.
\end{multline*}
Integrating both sides over $\C^n$ with respect to ${\mathrm d}v$ and applying Fatou's lemma once more gives the estimates
\begin{equation}\label{d-bar-2}
\begin{split}
	\|f_{m, 1}- f_{ 1}\|_{\IDA^{s, q}}^s 
	&\le C \int_{\Cn}\left \|(f_{m, 1}- f_{ 1}) - {\mathrm P} _{z, r} (f_{m, 1}- f_{ 1})\right \|_{L^q(B(z, r/2))}^s{\mathrm d}v\\
	&\le C\int_{\Cn} \liminf_{k\to \infty} G_{q, 2r}( f_{m, 1} -f_{k, 1})^s {\mathrm d}v\\
	&\le  C \liminf_{k\to \infty}\|f_{m, 1} -f_{k, 1}\|^{s}_{\IDA^{s, q}}.
\end{split}
\end{equation}
Therefore, $\lim_{m\to \infty} \|f_{m, 1}- f_1\|_{\IDA^{s, q}}=0$. Set $f=f_1+f_2\in L^q_{\mathrm{loc}}$. From (\ref{f-m-2}) and (\ref{d-bar-2}) it follows that
$$
	\lim_{m\to \infty} \| f_{m}-f \|_{\IDA^{s, q}} 
	\le \lim_{m\to \infty}\left( \| f_{m, 1}-f_1\|_{\IDA^{s, q}}+ \| f_{m, 2}-f_2\|_{\IDA^{s, q}}\right)=0,
$$
which completes the proof of the completeness and of the theorem.
\end{proof}

\begin{corollary}\label{bda-vad}
Let $1\le q<\infty$. With the norm induced by $\|\cdot\|_{\BDA^q}$, the quotient space $\mathrm{BDA}^q/H(\Cn)$ is a Banach space and $\mathrm{VDA}^q$
is a closed subspace of $\mathrm{BDA}^q$.
\end{corollary}

\begin{proof}
The proof of Theorem \ref{complete} works for $s=\infty$, so $\mathrm{BDA}^q/H(\Cn)$ is a Banach space in $\|\cdot\|_{\mathrm{BDA}^q}$.  That $\mathrm{VDA}^q$
is a closed subspace of $\BDA^q$ can be proved in a standard way.
\end{proof}

\section{Proof of Theorem~\ref{main1}}\label{hankel-A}

Given two $F$-spaces $\mathrm X$ and $\mathrm Y$, we write $B(\mathrm X)$ for the unit ball of $\mathrm X$.
 A linear operator $\mathrm T$ from $\mathrm X$ to $\mathrm Y$ is bounded (or compact) if ${\mathrm T}(B(\mathrm X))$ is bounded (or relatively compact) in $\mathrm Y$. The collection of all bounded  (and compact)  operators from $\mathrm X$ to $\mathrm Y$ is denoted by ${\mathcal B}(  \mathrm X, \mathrm Y)$ (and by ${\mathcal K}(\mathrm X, \mathrm Y)$ respectively). We use $\|{\mathrm T}\|_{{\mathrm X}\to {\mathrm Y}} $ to denote the corresponding  operator norm. In particular, we recall that when $0<p<1$, the Fock space $F^p(\varphi)$ with the metric given by $d(f,g)=\|f-g\|^p_{p, \varphi}$ is an $F$-space.

To deal with the boundedness and compactness of Hankel operators, we need an additional result involving positive measures and their averages. More precisely, given a positive Borel measure $\mu$ on $\Cn$ and $r>0$, we write $\widehat{\mu}_r(z)=\mu(B(z, r))$. Notice, in particular, $\widehat{\mu}_r$ is a constant multiple of the averaging function induced by the measure $\mu$.

\begin{lemma}\label{integral-est}  Suppose $0<p\le1$ and $r>0$. There is a constant $C$ such that, for  $\mu$  a positive Borel measure on $\Cn$, $\Omega$  a domain in $\Cn$, and $g\in H(\Cn)$, it holds that
$$
	\left(\int_{\Omega} \left | g(\xi) e^{-\varphi(\xi)}\right| d\mu(\xi)\right)^p
	\le  C \int_{\Omega^+_r} \left | g(\xi) e^{-\varphi(\xi)}\right|^p \widehat{\mu}_r(\xi)^p dv(\xi),
$$
where $\Omega^+_r = \bigcup_{\{z\in \Omega\}} B(z, r)$.
\end{lemma}

\begin{proof}
Let $\{a_j\}_{j=1}^\infty$ be an $r/4$-lattice. Notice that
$$
	\widehat{\mu}_{r/4}(a_j)
	\le C \inf_{w\in B(a_j, r/2)} \widehat{\mu}_{r}(w)
$$
for all $j\in\N$ and $(a+b)^p\le a^p+b^p$ for $a, b\ge 0$. Then
\beqm
	&& \left(\int_{\Omega} \left | g(\xi) e^{-\varphi(\xi)}\right| d\mu(\xi)\right)^p\\
	&\le&  \sum_{j=1}^\infty \left(\int_{B(a_j, r/4)\cap \Omega} \left | g(\xi) e^{-\varphi(\xi)}\right| d\mu(\xi) \right)^p\\
	&\le& C \sum_{\{j:\, B(a_j, r/4)\cap \Omega \neq \emptyset\} }   \sup_{\xi\in B(a_j, r/4)\cap \Omega} \left | g(\xi) e^{-\varphi(\xi)}\right|   ^p \widehat{\mu}_{r/4}(a_j)^p\\
	&\le& C \sum_{\{j:\, B(a_j, r/4)\cap \Omega \neq \emptyset\} }   \widehat{\mu}_{r/4}(a_j)^p \int_{B(a_j, r/2 )} \left | g(\xi) e^{-\varphi(\xi)}\right| ^p    dv(\xi)\\
	&\le& C \sum_{\{j:\, B(a_j, r/4)\cap \Omega \neq \emptyset\} } \int_{B(a_j, r/2 )} \left | g(\xi) e^{-\varphi(\xi)}\right| ^p \widehat{\mu}_{r }(\xi)^p   dv(\xi)\\
	&\le& C  \int_{\Omega^+_r } \left | g(\xi) e^{-\varphi(\xi)}\right| ^p \widehat{\mu}_{r}(\xi) ^p  dv(\xi),
\eqm
which completes the proof.
\end{proof}

\begin{remark}\label{compactness remark} To prove compactness of Hankel operators on spaces that are not necessarily Banach spaces, we use the following result. For $0<p, q<\infty$, $H_f : F^p(\varphi)\to L^q(\varphi)$ is compact if and only if
$$
	\lim_{m\to \infty}\|H_f(g_m)\|_{q, \varphi} =0
$$
for any sequences $\{g_m\}_{m=1}^\infty$ in $B(F^p(\varphi))$ satisfying
$$
	\lim_{m\to \infty} \sup_{w\in E}|g_m(w)|=0
$$
for compact subsets $E$ in $ \C^n$.

Necessity is trivial. To prove sufficiency, we notice that $B(F^p(\varphi))$ is a normal family, so for any sequence $\{g_m\}_{m=1}^\infty \subset B(F^p(\varphi))$, there exist a holomorphic function $g_0$ on $\C^n$ and a subsequence $ \{g_{m_j}\}_{j=1}^\infty$ such that
$$
	\lim_{j\to \infty} \sup_{w\in E}|g_{m_j}(w)-g_0(w)|=0.
$$
This and Fatou's Lemma imply that $g_0\in B(F^p(\varphi))$, and hence by the hypothesis, we get
$$
	\lim_{j\to \infty} \| H_f(g_{m_j})-H_f(g_0)\|_{q, \varphi}
	= \lim_{j\to \infty} \| H_f(g_{m_j}-g_0)\|_{q, \varphi}=0.
$$
Therefore, $H_f\left(B(F^p(\varphi))\right)$ is sequential compact in $L^q(\varphi)$, that is, the Hankel operator $H_f:F^p(\varphi)\to L^q(\varphi)$ is compact.
\end{remark}

\subsection{The case $0<p\leq q<\infty$ and $q\geq 1$}\label{s:4.1}
\begin{proof}[Proof of Theorem \ref{main1} {\rm (a)}]
By \eqref{basic-est-a}--\eqref{nor-bergman},
\begin{equation}\label{norm-kernel}
 \|k_z\|_{p, \varphi} \le C, \,\,  \sup_{\xi\in B(z, r_0)} |k_z(\xi)| e^{-\varphi(\xi)}\ge C\ {\rm and }\,
  \lim_{z\to \infty}  \sup_{w\in E}|k_z(w)| =0
\end{equation}
  for any compact subset $E\subset \Cn$.
As in the proof of Theorem~4.2 of \cite{HL19}, there is an $r_0$ such that for all $z\in \C^n$, we have
\begin{equation}\label{hankel-on-k}
\begin{split}
	\| H_f(k_z)\|_{q, \varphi}^q &\ge   \int_{B(z, r_0)} \left
	|f k_z -P(fk_z)   \right|^q e^{-q\varphi }dv\\
	&\ge C\fr 1{|B(z, r_0)|} \int_{B(z, r_0)} 
	\left |f  - \fr 1 {k_z}P(fk_z)  \right|^q dv
	\ge C G_{q, r_0}^q (f)(z).
\end{split}
\end{equation}
If  $H_f\in {\mathcal B}(F^p(\varphi),\, L^q(\varphi))$,
\begin{equation}\label{p-less-than-q-1}
\left\|   f \right\|_{\mathrm{BDA}^q}\le C \|H_f\|_{F^p(\varphi)\to L^q(\varphi)}<\infty;
\end{equation}
if  $H_f\in {\mathcal K}(F^p(\varphi),\, L^q(\varphi))$, then $f\in \mathrm{VDA}^q$ because
\begin{equation}\label{compact-G}
 \lim_{z\to \infty} G_{q, r_0}^q (f)(z) \le C    \lim_{z\to \infty} \|H_f(k_z)\|_{q, \varphi}=0.
\end{equation}

Next we prove sufficiency.  Suppose that $f\in \mathrm{BDA}^q$ and decompose  $f=f_1+f_2$ as in \eqref{decomposition-a}. Write $d\mu= |f_2|^q dv$ and $d\nu= |\overline{\partial}f_1| ^q   dv$.  According to Theorem~2.6 of~\cite{HL14} and Corollary~\ref{VDA-1}, both $d\mu$ and $d\nu$ are $(p, q)$-Fock Carleson measures.
We claim that both $f_1 , f_2\in \mathcal S$. Indeed, since $q\ge 1$, we can use Lemma \ref{integral-est} with $\Omega= \C^n$ and the measure $|f_2|dv$ to get
\begin{equation}\label{S-class}
\begin{split}
	\int_{\Cn}  |f_2(\xi) K(\xi, z)| e^{-\varphi(\xi)}dv(\xi)
	&\le C \int_{\Cn}  M_{1, r}( f_2 ) (\zeta) |K (\zeta, z) |e^{-\varphi(\zeta)}dv(\zeta)\\
	&\le C \int_{\Cn}  M_{q, r}( f_2 ) (\zeta) |K (\zeta, z) |e^{-\varphi(\zeta)}dv(\zeta).
 \end{split}
\end{equation}
Since $f\in \mathrm{BDA}^q$, Lemma \ref{basic-docomp} implies that
$$
	\int_{\Cn}  |f_2(\xi) K(\xi, z)| e^{-\varphi(\xi)}dv(\xi)
	\le C \|f\|_{BDA^q} \int_{\Cn}  | K(\xi, z)| e^{-\varphi(\xi)}dv(\xi) <\infty
$$
for $z\in \C^n$. Hence, $f_2\in \mathcal S$, and
so also $f_1=f-f_2\in \mathcal S$ because $f\in \mathcal S$ by the hypothesis. Since the Bergman projection $P$ is bounded on $L^q(\varphi)$ when $q\ge 1$, we have for $g\in\Gamma$,
\begin{align*}
	\| H_{f_2}(g)\|_{q, \varphi}
	&\le  (1+\|P\|_{L^q(\varphi)\to F^q(\varphi) }) \|f_2 g\|_{q, \varphi}\\
	&\le C \|M_{q, r}(f_2)\|_{L^\infty} \|g\|_{q, \varphi}
	\le C \|M_{q, r}(f_2)\|_{L^\infty} \|g\|_{p, \varphi},
\end{align*}
where the second inequality follows from Lemma~\ref{integral-est}. For $H_{f_1}(g)$ with $g\in \Gamma$, Corollary~\ref{d-bar-hankel} shows that $H_{f_1}(g)= A_\varphi(g\overline \partial f_1) -P(A_\varphi(g\overline \partial f_1))$. Lemma \ref{hankel-and-d-bar} implies
\begin{equation}\label{f_1-estimate}
	\|  H_{f_1}(g)\|_{q, \varphi}
	\le C \|g \, |{\overline \partial} f_1|\|_{q, \varphi}\le C \|{\overline \partial} f_1\|_{L^\infty} \| g \|_{q, \varphi}\le C\|{\overline \partial} f_1\|_{L^\infty} \| g \|_{p, \varphi}.
\end{equation}
From the above estimates and the fact that $\Gamma$ is dense in $F^p(\varphi)$, it follows that for $0 < p\le q< \infty$, we have
\begin{equation}\label{p-less-than-q-a}
	\|H_{f}\|_{F^p(\varphi)\to L^q(\varphi)}
	\le C \left\{  \|{\overline \partial} f_1\|_{L^\infty} + \|M_{q, r}(f_2)\|_{L^\infty} \right\}
	\le C \|f\|_{\mathrm{BDA}^q}
\end{equation}
where the latter inequality follows from Lemma~\ref{basic-docomp}.

For compactness, suppose $f\in \mathrm{VDA}^q$ so that  $f=f_1+f_2$  is as (\ref{decomposition-a}).
Notice that both $d\mu=|f_2|^q   dv$   and $d\nu= |\overline{\partial}f_1| ^q   dv$ are   vanishing $(p, q)$-Fock Carleson measures. Let $\{g_m\}$ be a bounded sequence   in $F^p(\varphi)$ converging to zero uniformly on compact subsets of $\Cn$. Then
\begin{align*}
	\|H_{f_2}(g_m)\|_{L^q(\varphi)}
	&\le \|g_m f_2\|_{q, \varphi} +  \|P(g_m f_2)\|_{q, \varphi}\\
	&\le  C \left (\int_{\C}\left |g_m  e^{- \varphi}\right|^q d\mu \right)^{\frac{1}{q}} \to 0
\end{align*}
as $m\to \infty$. To prove $ \lim_{m\to \infty} \|H_{f_1}(g_m)\|_{L^q(\varphi)}  =0$, for each $m$ we pick some $g_m^*\in \Gamma$ so that $\|g_m-g_m^*\|_{p, \varphi}<\fr 1m$. Clearly, $\{g_m^*\}_{m=1}^\infty$ is bounded in $F^p(\varphi)$, and $\lim_{z\to \infty}  \sup_{w\in E}|g_m^*(w)| =0$ for any compact subset $E$. Again by Corollary \ref{d-bar-hankel},
$$
	\|H_{f_1}(g_m^*)\|_{L^q(\varphi) }
	\le C\left\|g_m^*   \overline{\partial}f_1 \right \|_{L^q(\varphi)}
	\le C \| g_m^*\|_{L^q(\Cn, d\nu)} \to 0 \, \textrm{ as }\, m\to \infty.
$$
Therefore, since (\ref{basic-docomp}) guarantees that  $H_{f_1}\in {\mathcal B}(F^p(\varphi), L^q(\varphi))$, it follows that $\lim_{m\to \infty} \| H_{f_1}(g_m)\|_{L^q(\varphi)}  =0$, and so
$$
	H_f=H_{f_1} + H_{f_2}\in {\mathcal K}(F^p(\varphi), L^q(\varphi)).
$$

Finally, it remains to notice that the norm equivalence \eqref{bounded-g} follows from \eqref{p-less-than-q-1} and \eqref{p-less-than-q-a}.
 \end{proof}

\subsection{The case $1\le q<p<\infty$}

We can now prove the case $q<p$ under the assumption that $q\ge 1$.

\begin{proof}[Proof of Theorem \ref{main1} {\rm (b)}] Suppose that $H_f\in {\mathcal B}\left(F^p(\varphi), L^q(\varphi)\right)$.
Because the proof of sufficiency is similar to the implication $(A)\Rightarrow (C)$ of Theorem 4.4 in~\cite{HL19}, we only give the sketch here.

Indeed, take $r_0$ as in~\eqref{norm-kernel}, and set $t= r_0/4$. Let $\{a_j\}_{j=1}^\infty$ be a $t/2$-lattice. By Lemma 2.4 of \cite{HL14}, $\left\| \sum_{j=1}^\infty \lambda_j k_{a_j}   \right\|_{p, \varphi}\le C \|\{\lambda_j\}\|_{l^p}$ for all $\{\lambda_j\}_{j=1}^\infty \in l^p$, where the constant $C$ is independent of $\{\lambda_j\}_{j=1}^\infty$. Let $\{\phi_j\}_{j=1}^\infty$  be the sequence   Rademacher functions on the interval $[0, 1]$. Using the boundedness of $H_f$, we get
\begin{equation}\label{q<p--a}
	\left\|  H_f\left(  \sum_{j=1}^\infty \lambda_j\phi_j(s)  k_{a_j}(\cdot) \right)\right\|_{q, \varphi}
	\le C \|H_f\|_{F^p(\varphi)\to L^q(\varphi)} \left\| \{|\lambda_j|^q \}\right\|_{l^{\frac p q}}^{\frac 1q}
\end{equation}
for $s\in [0, 1]$. On the other hand,
\begin{equation}\label{q<p--b}
	\int_{B(a_j, t)} \left|H_f(k_z)(\xi) e^{-\varphi(\xi)} \right|^q dv(\xi)
	\ge C G_{q, t}(f)(a_j)^q.
\end{equation}
This and Khintchine's inequality yield
$$
	\int_0^1 \left\|  H_f\left(  \sum_{j=1}^\infty \lambda_j\phi_j(s)  k_{a_j}(\cdot) \right)\right\|_{q, \varphi}^q dt
	\ge C \sum_{j=1}^\infty |\lambda_j|^q G_{q, t}(f)(a_j)^q.
$$
Combining this with \eqref{q<p--a} gives
$$
	\sum_{j=1}^\infty |\lambda_j|^q G_{q, t}(f)(a_j)^q
	\le C \|H_f\|_{F^p(\varphi)\to L^q(\varphi)}^q \left\| \{|\lambda_j|^q \}\right\|_{l^{\frac p q}}
$$
for all  $\{|\lambda_j|^q \}_{j=1}^\infty \in l^{\frac p q}$.  By duality with the exponentials $\frac p q$ and its conjugate,
$$
	\sum_{j=1}^\infty G_{q, t}(f)(a_j)^{\frac{pq}{p-q}}
	\le C \|H_f\|_{F^p(\varphi)\to L^q(\varphi)}^{\frac {pq}{p-q}}.
$$
Therefore, by \eqref{P-z-r},
\begin{equation} \label{q-lessthen-p-1}
\begin{split}
	\int_{\C^n} G_{q, t/2}(f)(z)^{\frac{pq}{p-q}}dv(z)
	&\le \sum_{j=1}^\infty\int_{B(a_j, t/2)} G_{q, t/2}(f)(z)^{\frac{pq}{p-q}}dv(z)\\
	&\le C  \|H_f\|_{F^p(\varphi)\to L^q(\varphi)}^{\frac {pq}{p-q}},
\end{split}
\end{equation}
which means that $f\in \mathrm{IDA}^{s, q}$ with the estimate $\left\|   f \right\|_{\mathrm{IDA}^{s, q}}\le C \|H_f\| $.

It should be pointed out that the right hand side of the estimate (4.24) (the analog of \eqref{q-lessthen-p-1} above) in~\cite{HL19} should read $C \|H_f\|_{A^p_\omega \to L^q_\omega}^{\frac {pq}{p-q}}$, and not $C\|H_f\|_{A^p_\omega\to L^q_\omega}$ as stated there.

Conversely, suppose $ f\in \mathrm{IDA}^{s, q}$. As before, decompose $f=f_1+f_2$ as in \eqref {decomposition-a}. From Lemma  \ref{basic-docomp} we know that $\left\| M_{q, r}(f_2) \right\|_{\fr {pq}{p-q}} \le C \left\|   f \right\|_{\mathrm{IDA}^{s, q}}
$. Applying H\"{o}lder's inequality to the right hand side integral in \eqref{S-class} with exponent $\frac{pq}{p-q}$ and its conjugate exponent $t$, since we have $\|K(\cdot, z)\|_{t, \varphi}<\infty$, it follows that
\begin{equation*}
\begin{split}
    \int_{\Cn} &|f_2(\xi) K_z(\xi)| e^{-\varphi(\xi)}dv(\xi)\le C \, \left\| M_{q, r}(f_2) \right\|_{\fr {pq}{p-q}} \cdot \|K_z\|_{t, \varphi}<\infty.
\end{split}
\end{equation*}
This implies $f_2\in \mathcal S$, and so also $f_1 \in \mathcal{S}$.

Now for $d\nu= |\overline{\partial}f_1| ^q dv$, applying H\"{o}lder's inequality again with $\fr p{p-q}$ and its conjugate exponent $p/q$,  we get
\begin{equation}\label{p>q--m}
\begin{split}
	\left\| \widehat{\nu}_r \right\|_{  L^{\fr p{p-q}}}^{\fr p{p-q}} 
	&= C \int_{\Cn }\left\{ {\int_{B(\xi, r)}   |{\overline \partial} f_1(\zeta)| ^q dv(\zeta)} \right\}^{\frac p{p-q}}  dv(\xi)\\
	&\le C \int_{\Cn} dv(\xi)   \int_{B(\xi, r)}   |{\overline \partial} f_1(\zeta)| ^{\fr {pq}{p-q}} \,dv(\zeta)\\
	&\simeq C\int_{\C^n}|{\overline \partial} f_1(\zeta)| ^{\frac{pq}{p-q}} dv(\zeta)<\infty.
\end{split}
\end{equation}
Theorem~2.8 of~\cite{HL14} shows that $\nu$ is a vanishing $(p,q)$-Fock Carleson measure, that is, the multiplier $M_{f_1}: g\mapsto g|{\overline \partial} f_1|$  is compact from  $F^p(\varphi)$ to $L^q(\varphi)$ (see Proposition~\ref{F-C-cpt}). Therefore, by Lemma~\ref{hankel-and-d-bar} (A), $A_\varphi(\cdot \,  \overline \partial f_1 )$ is compact from $F^p(\varphi)$ to $L^q(\varphi)$. Moreover, $\Gamma $ is dense  in $F^p(\varphi)$ and, by  Corollary~\ref{d-bar-hankel}, $H_{f_1}(g)= A_\varphi(g \, \overline \partial f_1 )-P\circ  A_\varphi(g \, \overline \partial f_1 )$ for  $g\in \Gamma $.  Hence,   $H_{f_1}:   F^p(\varphi) \to L^q(\varphi)$ is compact and we obtain the norm estimate
\begin{equation}\label{hankel-f-1}
	\|H_{f_1}\|_{F^p(\varphi)\to L^q(\varphi)} 
	\le C \sup_{\{g\in F^p(\varphi):\, \|g\|_{p,\varphi}\le 1\}} \|A_{\varphi}(g \overline \partial f_1)\|_{q, \varphi} 
	\le C \|{\overline \partial} f_1\|_{\fr {pq}{p-q}}.
\end{equation}
Similarly to \eqref{p>q--m}, using Lemma \ref{basic-docomp}, for $d\mu= |f_2|^q dv$, we get
\begin{align*}
	\left\| \widehat{\mu}_r \right\|_{  L^{\fr p{p-q}}}^{\frac p{p-q}} 
	&=  C \int_{\Cn }\left\{ {\int_{B(\xi, r)}   | f_2(\zeta)| ^q dv(\zeta)}\right\}^{\fr p{p-q}}  dv(\xi)\\
       &=C\left\| M_{q, r}(f_2) \right\|_{\fr {pq}{p-q}}^{\fr {pq}{p-q}}\le C \left\|   f \right\|_{\mathrm{IDA}^{s, q}}^{s}<\infty.
\end{align*}
Hence, $d\mu= |f_2|^q  dv$ is a vanishing $(p, q)$-Fock Carleson measure. It follows from Proposition~\ref{F-C-cpt} that the identity operator
$$
	{\mathrm{I}}: F^p(\varphi) \to L^q(\Cn,  e^{-q\varphi} d\mu)
$$ 
is compact. Using the inequality
\begin{equation}\label{q-lessthen-p-8}
	\|H_{f_2}(g)\|_{q, \varphi}
	\le  C \| f_2 g\|_{q, \varphi}
	=C \|{\mathrm{I}}(g)\|_{L^q(\C,\, e^{-q\varphi} d\mu)},
\end{equation}
we see that $H_{f_2}$ is compact from $F^p(\varphi)$ to $L^q(\varphi)$. 

It remains to notice that the norm equivalence in~\eqref{q-lessthen-p-a} follows from combining the estimates in \eqref {q-lessthen-p-1}, \eqref{hankel-f-1}, and \eqref{q-lessthen-p-8}.
\end{proof}

\begin{remark}
In~\cite{St92}, it was proved that for bounded symbols $f$, the Hankel operator $H_f : F^2 \to L^2$ is compact if and only if
\begin{equation}\label{e:SZ}
	\|(I-P)(f\circ \phi_\lambda)\| \to 0
\end{equation}
as $|\lambda|\to \infty$, where $\phi_\lambda(z) = z+\lambda$.
This characterization was recently generalized to $F^p_\alpha$ with $1<p<\infty$ in~\cite{HV19}. Here we note that, using a generalization of Lemma~8.2 of~\cite{Zh12} to the setting of $\C^n$, one can prove that Stroethoff's result remains true for Hankel operators acting from $F^p_\alpha$ to $L^q_\alpha$ whenever $1\leq p, q<\infty$ even for unbounded symbols.
\end{remark}

\subsection{The case $0<p\le q\le 1$ with bounded symbols} We start with the following preliminary lemma whose proof can be completed with a standard $\varepsilon$ argument.

\begin{lemma}\label{norm-0}
Suppose that $0<p<\infty$,  $h\in L^\infty$ and  $\lim_{z\to \infty} h(z)=0$. Then for any bounded sequence $\{ g_j\}_{j=1}^\infty$ in $L^p_\varphi$ satisfying $\lim_{j\to \infty} g_j(z)=0$ uniformly on compact subsets of $\Cn$, it holds that $\lim_{j\to \infty} \|g_j h\|_{p, \varphi} =0$.
\end{lemma}

\begin{proof}
If $R$ is sufficiently large, there is a $C>0$ such that
\begin{align*}
	\|g_j h\|_{p, \varphi}^p &= \left( \int_{B(0, R)}+ \int_{\Cn \setminus B(0, R)}\right) |g_j(\xi)h(\xi)e^{-\varphi(\xi)}|^pdv(\xi)\\
	&\le \|h\|^p_{L^\infty} \sup_{|\xi|\le R} |g_j(\xi)e^{-\varphi(\xi)}|^p + C \|g_j\|_{p, \varphi}^p \to 0
\end{align*}
as $j\to \infty$.
\end{proof}

\begin{proof}[Proof of Theorem \ref{main1} {\rm (c)}]
Suppose that $f\in \mathcal S$. Then $f\in L^q_{\mathrm {loc}}$ for $0<q\le 1$, and we may decompose    $f=f_1+f_2$   as in~\eqref{decomposition-a} with $t=  r/2$. We claim that, for $g\in \Gamma$,
\begin{equation}\label{f-1-estimath}
	\|H_{f_1}(g)\|^q_{q, \varphi} 
	\le C \int_{\C^n} \left | g(\xi)e^{-\varphi(\xi)}\right|^q
	\left  \|  {\overline \partial}f_1   \right \|_{L^\infty(B(\xi,  r), dv)}^q  dv(\xi)
\end{equation}
and
\begin{equation}\label{f-2-estimath}
	\| H_{f_2}(g)\|^q_{q, \varphi} 
	\le C \int_{\Cn}\left | g(\xi)e^{-\varphi(\xi)}
	\right|^q  M_{1, r}(f_2)(\xi) ^q dv(\xi).
\end{equation}
To estimate  $\|H_{f_1}(g)\|_{q, \varphi}$, we use the representation
$$
	H_{f_1}(g)= A_{\varphi} (g\overline \partial f_1)-P(A_{\varphi} (g\overline \partial f_1)),
$$
(see~\eqref{d-bar-s}), which suggests that we define a measure $d\mu_z$ as follows
$$
	d\mu_z(\xi) =  \left|{\overline \partial}f_1(\xi)\right| \left\{\fr 1{|\xi-z|} +\fr 1{|\xi-z|^{2n-1}} \right\}e^{- m|\xi-z|} dv(\xi).
$$
Then there is a constant $C$ such that, for $w\in \Cn$,
$$
	\int_{B(w, r)}  \left|{\overline \partial}f_1(\xi)\right| \left\{\fr 1{|\xi-z|} +\fr 1{|\xi-z|^{2n-1}} \right\}e^{- m|\xi-z|^2} dv(\xi)
	\le C \int_{B(w, r)} d\mu_z(\xi).
$$
 Also, it is easy to verify that
$$
	\widehat{(\mu_z)}_r(w) 
	\le C \sup_{\eta\in B(w, r)}\left|{\overline \partial}f_1(\eta)\right| e^{- m|w-z|},
$$
where the constant $C$ is independent of $z, w\in \Cn$. Recall that
\begin{multline*}
	A_\varphi(g \overline \partial f_1 )(z)
	= \int_{\Cn} e^{\langle 2 \partial \varphi, z-\xi  \rangle }\\
	\times\sum_{j<n } g(\xi){\overline \partial}f_1(\xi)\wedge \frac{\partial|\xi-z|^2\wedge (2 {\overline \partial} \partial\varphi(\xi))^j \wedge ( {\overline \partial} \partial |\xi-z|^2)^{n-1-j} }{j!|\xi-z|^{2n-2j}}.
\end{multline*}
Therefore, using~\eqref{d-bar-s-V1} and Lemma \ref{integral-est}, we get
\begin{equation}\label{P-A-estimate}
\begin{split}
	\left|A_\varphi(g  \overline \partial f_1 )(z) e^{-\varphi(z)} \right|^q
	&\le   C \left(  \int_{\C^n} \left | g (\xi)e^{-\varphi(\xi)}\right|d\mu_z(\xi)\right)^q\\
	&\le   C\int_{\C^n}\left | g (\xi)e^{-\varphi(\xi)}\right|^q\left\|  {\overline \partial}f_1   \right \|_{L^\infty(B(\xi, r), dv)}^q  e^{- q m|\xi-z|} dv(\xi).
\end{split}
\end{equation}
Fubini's theorem yields
\begin{equation}\label{f-1-estimate}
\begin{split}
	&\|A_\varphi(g  \overline \partial f_1 )  \|_{q, \varphi}^q\\
	&\le  C \int_{\Cn} dv(z) \int_{\C^n} \left | g (\xi)e^{-\varphi(\xi)}\right|^q \left  \|  {\overline \partial}f_1   \right \|_{L^\infty(B(\xi,  r), dv)}^q e^{- qm |\xi-z|}  dv(\xi)\\
	&\le C \int_{\C^n}\left | g (\xi)e^{-\varphi(\xi)}\right|^q \left \|  {\overline \partial}f_1 \right \|_{L^\infty(B(\xi,  r), dv )}^q dv(\xi).
\end{split}
\end{equation}
To deal with $P\left(A_\varphi(g \overline \partial f_1 )\right)$, we use Lemma \ref{basic-est} to obtain positive constants $\theta$ and $C$ so that,  for $z\in \Cn$, we have
\begin{align*}
	&\int_{\Cn} |K(w, z)|e^{-m |\xi-z|} e^{-\varphi(z)} dv(z) \\
	&\le C e^{ \varphi(w)} \int_{\Cn} e^{-m|\xi-z| } e^{-\theta |w-z|} dv(z)\\
	&=  C e^{ \varphi(w)}\left( \int_{\{z: |z-\xi|\ge|z-w|\}} +\int_{\{z: |z-\xi|<|z-w|\}} \right)e^{-m|w-z| } e^{-\theta |\xi-z|} dv(z)\\
	&\le C e^{ \varphi(w)} e^{-\tau |\xi-w|},
\end{align*}
where $\tau=\min\{\theta, m\}$. Therefore, (\ref{P-A-estimate}) and   Fubini's theorem yield
\begin{align*}
	&\left|P\left (A_\varphi(g  \overline \partial f_1 )\right)(w)\right|\\
	&\le  C \int_{\Cn} \left | g (\xi) e^{-\varphi(\xi)}\right| \left\|  {\overline \partial}f_1   \right \|_{L^\infty(B(\xi,  r/2), dv)} dv(\xi)\\
	&\qquad\qquad\times\int_{\Cn} |K(w, z)|e^{-\theta |\xi-z|} e^{-\varphi(z)} dv(z)\\
	&\le  C e^{\varphi(w)} \int_{\Cn} \left | g  (\xi)e^{-\varphi(\xi)}\right|\left\|  {\overline \partial}f_1   \right \|_{L^\infty(B(\xi,  r/2 ), dv)}  e^{-\tau|\xi-w|}dv(\xi).
\end{align*}
Lemma  \ref{integral-est} again gives
$$
	\left\|P\left (A_\varphi(g  \overline \partial f_1 )\right)(w)\right\|_{q, \varphi}^q 
	\le C  \int_{\C^n} \left | g (\xi)e^{-\varphi(\xi)}\right|^q
       \left\|{\overline \partial}f_1 \right \|_{L^\infty(B(\xi, r), dv)}^q  dv(\xi).
$$
Combining this and \eqref{f-1-estimate}, we get \eqref{f-1-estimath}.

For \eqref{f-2-estimath}, notice first that
\begin{equation}\label{f-2-multip}
	\|f_2 g\|_{q, \varphi}^q
	\le C\int_{\C^n} \left | g(\xi)e^{-\varphi(\xi)}\right|^q M_{q, r}^q (f_2)(\xi) dv(\xi),
\end{equation}
and, by Lemma \ref{integral-est} with the measure $M_{1, r/2}(f_2) dv$, we have
\begin{equation}\label{q<1-estimate}
\begin{split}
	|P(f_2 g)(z)|^q 
	&\le C \left(  \int_{\Cn} \left|g(\xi)K(z, \xi) e^{-2\varphi(\xi)}\right|  M_{1, r/2}(f_2)(\xi) dv(\xi) \right)^q\\
	&\le C   \int_{\Cn} \left|g(\xi)K(z, \xi) e^{-2\varphi(\xi)}\right|^q  M_{1, r}(f_2 )(\xi)^q  dv(\xi).
\end{split}
\end{equation}
Integrating both sides of \eqref{q<1-estimate} against $e^{-q\varphi } dv$ over $\C^n$ and using~\eqref{nor-bergman}, we get
\begin{equation}\label{f-2-estimath-part2}
	\|P(f_2 g)\| _{q, \varphi}^q
	\le C   \int_{\C^n} \left|g(\xi) e^{-\varphi(\xi)}\right|^q  M_{1, r}(f_2 )(\xi)^q dv(\xi).
\end{equation}
This and \eqref{f-2-multip} imply \eqref{f-2-estimath}.

Now we suppose that $f\in L^\infty$ and  $0<p\le q<1$. For $g\in H(\Cn)$, similarly to the proof of~\eqref{f-2-estimath}, we have
$$
	\| H_{f}(g)\|_{q, \varphi}
	\le C \left( \int_{\Cn}\left | g(\xi)e^{-\varphi(\xi)}\right|^q  M_{1, r}(f )(\xi) ^q dv(\xi)\right)^{\fr 1q}
	\le C \|f\|_{L^\infty} \|g\|_{p, \varphi}.
$$
This implies boundedness of $H_f$ with the norm estimate \eqref{q-is-small}.

For the second assertion,  suppose first that $\lim_{|z|\to \infty} G_{q,r}(f)(z)=0$ for some $r>0$ and write $f=f_1 + f_2$ as above.
Since the unit ball $B(F^p(\varphi))$ of $F^p(\varphi)$ is a normal  family,  to show that $H_f$ is compact from $F^p(\varphi)$ to $L^q (\varphi)$, it suffices to prove that for $k=1, 2$,
$$
	\lim_{j\to \infty} \|H_{f_k}(g_j)\|_{q, \varphi}
	= \lim_{j\to \infty}\| f_k g_j-P (f_k g_j) \|_{q, \varphi}=0
$$
for any bounded sequence $\{g_j\}_{j=1}^\infty$ in $F^p(\varphi)$ with the property that
$$
	\lim_{j\to \infty} \sup_{w\in E}|g_j(w)|=0
$$
for $E$ compact in $\Cn$. From the assumption that $\lim_{z\to\infty}  M_{q, r}(f_2)(z)=0$, it follows that $d\mu=|f_2|^q dv$ is a vanishing $(p, q)$-Fock Carleson measure (see Theorem~2.7 of~\cite{HL14} and Proposition~\ref{F-C-cpt}). Therefore, we get
$$
	\|f_2 g_j\|_{q, \varphi}  = \|g_j\|_{L^ q(\Cn, |f_2|^q dv) }\to 0 \, \textrm{ as } j\to \infty.
$$
Notice also that $\|g\|_{q, \varphi}\le C \|g\|_{p, \varphi}$ for $g\in F^q(\varphi)$ and $p\le q$. Further, by \eqref{f-2-estimath}, we obtain
$$
	M_{1, r}(f_2)(\xi)\le \|f_2\|_{L^\infty}^{ 1-q } M_{q, r}(f_2)(\xi)^q,
$$
and applying Lemma \ref{norm-0} to $h=M_{q, r}(f_2)^{q^2}$, we get
\begin{align*}
	\|H_{f_2}g_j\|_{q, \varphi}^q 
	&\le C   \int_{\Cn}\left | g_j(\xi)e^{-\varphi(\xi)}\right|^q  M_{1, r}(f_2)(\xi) ^q dv(\xi)\\
	&\le C \|f_2\|_{L^\infty}^{(1-q)q } \int_{\Cn} \left |g_j(\xi)e^{- \varphi(\xi)} \right|^q M_{q, r}(f_2)(\xi)^{q^2} dv(\xi) \to 0
\end{align*}
as $j\to \infty$.  So $H_{f_2}\in \mathcal K\left(F^p(\varphi), L^q(\varphi)\right)$. As for $H_{f_1}$, it follows from Lemma~\ref{basic-docomp} that
$$
	\left  \|  {\overline \partial}f_1   \right \|_{L^\infty(B(\xi,  r), dv)}
	\le C G_{q, r} (f)(\xi)\to 0 \ \ \textrm{ when  }\ \xi\to \infty.
$$
Therefore, by~\eqref{f-1-estimath},
$$
	\|H_{f_1}(g_j )\|_{q, \varphi}^q 
	\le C \int_{\Cn} \left | g_j (\xi)e^{-\varphi(\xi)}\right|^q
	\left\|{\overline \partial}f_1   \right \|_{L^\infty(B(\xi,  r), dv)}^q  dv(\xi)\to 0
$$
as $j\to \infty$, and hence we have $H_{f_1}\in \mathcal K\left(F^p(\varphi), L^q(\varphi)\right)$.

Conversely, suppose that $H_f$ is compact from $F^p(\varphi)$ to $L^q(\varphi)$. Then, as in~\eqref{compact-G}, we have
\begin{equation}\label{IDA--q<1}
	\lim_{z\to \infty} G_{q, r}  (f)(z) 
	\le C    \lim_{z\to \infty} \|H_f(k_z)\|_{q, \varphi}=0
\end{equation}
for $r\in (0, r_0]$ fixed. We claim that \eqref{IDA--q<1} is valid for any $r>0$. To see this, we consider the Hankel operator $H_f$ on the Fock space $F^p_\alpha$. From \eqref{IDA--q<1}, using the sufficiency part, it follows that $H_f$ is compact from $F^p_\alpha$ to $L^q(\C^n, e^{-\frac {q\alpha} 2 |z|^2} dv)$.
Notice that the equality \eqref{classical} yields
$$
	\inf_{w\in B(z, r)}|K(w, z)|\ge C>0
$$
for any $r>0$ fixed, where the constant $C$ is independent of $z\in \C^n$. As in \eqref{hankel-on-k}, we have
$$
	\lim_{z\to \infty} G_{q, r}  (f)(z) 
	\le C    \lim_{z\to \infty} \|H_f(k_z)\|_{L^q\left (\C^n, e^{-\frac {q \alpha} 2 |z|^2} dv\right ) }=0.
$$
Thus, $f\in \VDA^q$. This completes the proof.
\end{proof}

The following Corollary \ref{0<q<1-G} is a direct consequence of the proof of Theorem \ref{main1} {\rm (c)} which we use to complement and extend the classical result of Berger and Coburn in the next section.

\begin{corollary} \label{0<q<1-G}
Suppose that $0<q<1$ and $f\in L^\infty$. Then the limit $\lim_{z\to \infty} G_{q, r}(f)(z)=0$ is independent of  $r>0$.
\end{corollary}

\section{Proof of Theorem \ref{main2}}\label{bdd-symbols}

\begin{proof}[Proof of the case $0<p\le q<\infty$]
For $R>0$,  let  $\left \{a_{k}\right \}_{k=1}^\infty$ be the $R/2$-lattice
$$
	\left\{\fr R{2 \sqrt{n}}  (m_1+k_1\im, m_2+k_2\im,\ldots, m_n+k_n\im)\in \Cn : m_j, k_j\in \mathbb{Z}, j=1, 2, \ldots, n\right\}.
$$
Choose $\rho\in C^\infty(\Cn)$ such that
$$
	0\le\rho\le 1, \,\,  \rho|_{B(0, 1/2)}\equiv 1,\,\, \mathrm{supp}\,  \rho \subseteq B(0, 3/4).
$$
Then $\left\|\nabla \rho\right\|_{L^\infty}<\infty$  and
$$
	0<\sum_{k=1}^\infty \rho( (z-a_k)/R)\le C
$$
for $z\in \Cn$.  Define  $\psi_{j, R}\in C^\infty(\C^n)$ by
$$
	\psi_{j, R}(z)=\frac{\rho( (z-a_j)/R)}{\sum_{k=1}^\infty \rho( (z-a_k)/R)}.
$$
Then $\{\psi_{j, R} \}_{j=1}^\infty$ is a partition of unity subordinate to $\{B(a_j, R)\}_{j=1}^\infty$ and
\begin{equation}\label{partition}
	R \left\|\nabla \psi_{j, R} (\cdot)\right\|_{L^\infty}\le C,
\end{equation}
where the constant $C$ is independent of $j$ and $R$.

Now we suppose that $f\in L^\infty$ and $H_f\in {\mathcal K}(F^p(\varphi), L^q(\varphi))$. Theorem \ref{main1} and Corollary \ref{0<q<1-G} imply that
\begin{equation}\label{G(f)--0}
	\lim_{z\to \infty} G_{q, 2R}(f)(z)  =0
\end{equation}
for $R>0$ fixed. As in~\eqref{h-z}, pick $h_{j, R}\in H(B(a_j, 2R))$  so that
\begin{equation}\label{average}
	\frac1{|B(a_{j }, 2R)|}\int_{B(a_{j }, 2R)}|f-h_{j, R}|^q dv = G_{q, 2R}(f)(a_{j})^q.
\end{equation}
By \eqref{bounded-h},
$$
	\sup_{z\in B(a_j, R)} |h_{j, R}(z)|\le C \|f\|_{L^\infty}.
$$
Set
$$
	f_{1, R}= \sum_{j=1}^\infty \psi_{j, R}\, h_{j, R} \, \, 
	\textrm{ and }\, \, f_{2, R}= f-f_{1, R}.
$$
From estimates \eqref{lattice} and \eqref{bounded-h}, it follows that there is a positive constant $C$ such that
\begin{equation}\label{f-j-bound-above}
        \|f_{1, R}\|_{L^\infty} +  \|f_{2, R}\|_{L^\infty}\le C\|f\|_{L^\infty}
\end{equation}
for $R>0$. Lemma \ref{basic-docomp} and \eqref{G(f)--0} imply that
$$
 	\lim_{z\to \infty}  M_{q, R}(\overline {f_{2, R}} )(z)
	= \lim_{z\to \infty}  M_{q, R}(f_{2, R})(z)=0,
$$
and so
\begin{equation} \label{H-f-2}
   H_{\overline {f_{2, R}}}\in {\mathcal K}(F^p(\varphi), L^q(\varphi)).
\end{equation}
Recall that $P_{z,  R}$ is the standard Bergman projection from $L^2(B(z,  R), dv)$ to $A^2(B(z,  R), dv)$. Since $h_{j, R}$ is bounded on $B(a_j,  R)$, we have $h_{j, R} = P_{a_j, R}(h_{j, R})$, that is,
$$
	\overline{ h_{j, R}(z)}= \fr 1{\pi} \int_{B(a_j,  R)}
	\frac{R^2 \overline{h_{j, R}(\xi)} dv(\xi)}{\left( R^2-  ( \xi-a_j)\cdot \overline{( z-a_j)}  \right)^{n+1}}, \quad z\in B(a_j, R).
$$
Hence,
\begin{equation}\label{f-1-bar}
	\left|{\overline\partial}\,\, \overline{ h_{j, R}(z)}\right|
	\le C \frac{\|h_{j, R}\|_{L^\infty( B(z,  R), dv)}}{R} \, \,
	\textrm{ for } \, z\in \overline{B(a_j, 3R/4)}.
 \end{equation}
Notice that $\mathrm{supp}\, \psi_{j, R}\,h_{j, R} \subseteq \overline { B(a_j, 3R/4)}$, and the estimates (\ref{partition}) and (\ref{f-1-bar}) imply that
\begin{equation*}
	\left|{\overline \partial}\, \overline{f_{1, R}} \right |
	\le  \sum_{j=1}^\infty \left| \left({\overline \partial} \psi_{j, R} \right) \overline {h_{j, R}}\right| + \sum_{j=1}^\infty \psi_{j, R} \left|  {\overline \partial}\left( \overline h_{j, R} \right) \right| 
	\le C \fr  {\|f\|_{L^\infty}} R.
\end{equation*}
Therefore, using \eqref{f_1-estimate} (when $q\ge1$) and \eqref{f-1-estimath} (when $q<1$), we have
\begin{equation*}
	\| H_{\overline{f_{1, R}}} \|_{F^p(\varphi)\to L^q(\varphi)}^p 
	\le C {\left\|{\overline \partial}\, \overline{f_{1, R}} \right \|_{L^\infty}  } \le C \frac{\|f\|_{L^\infty} } R.
\end{equation*}
The constants $C$ above are all independent of $f$ and $R$. Therefore,
$$
	\|H_{\overline f} -H_{\overline {f_{2, R}}}\|_{F^p(\varphi) \to L^q(\varphi)} 
	= \| H_{{\overline {f_{1, R}}}}\|_{F^p(\varphi) \to L^q(\varphi)}
	\le C \frac{\|f\|_{L^\infty} } R\to 0
$$
as $R\to \infty$. Finally, using \eqref{H-f-2} and the fact that ${\mathcal K}(F^p(\varphi), L^q(\varphi))$ is closed under the operator norm, we see that $H_{\overline f} \in {\mathcal K}(F^p(\varphi), L^q(\varphi))$, which completes the proof.
\end{proof}

To deal with the case $1\le q<p<\infty$, we use the Ahlfors-Beurling operator, which is a very well-known Cader\'{o}n-Zygmund operator on $L^p(\mathbb C)$, $1<p<\infty$, defined as follows
$$
	\mathfrak{T} (f)(z) = p. v.\, - \fr 1\pi \int_{\mathbb C} \fr {f(\xi)}{(\xi-z)^2} dv(\xi),
$$
where $p. v.$ means the Cauchy principal value. The Ahlfors-Beurling operator connects harmonic analysis and complex analysis, and it is of fundamental importance in several areas of mathematics including PDE and quasiconformal mappings.  See \cite{Ah06} and \cite{AIM09} for further details and examples.

\begin{lemma} \label{partial-derivatives} Suppose $1<s<\infty$. Then there is some constant $C$, depending only on $s$, such that, for $f\in C^2(\Cn)\cap L^\infty$ and $j=1, 2, \cdots, n$,
\begin{equation}\label{partial-dir}
  \left\| \fr {\partial f}{\partial z_j} \right\|_{L^s}\le C  \left\| \fr {\partial f}{\partial \overline z_j} \right\|_{L^s}.
\end{equation}
\end{lemma}

\begin{proof}
We consider the case $n=1$ first. Let $f\in C^2(\mathbb C)\cap L^\infty$. If $\left\| \fr {\partial f}{\partial \overline z} \right\|_{L^s}=0$, then $f\in H(\mathbb C)\cap L^\infty$, which implies that the function $f$ is constant and the estimate   $(\ref{partial-dir})$ follows. Next we suppose that $\left\| \fr {\partial f}{\partial \overline z} \right\|_{L^s}>0$. Take $\psi(r)\in C^\infty(\mathbb R)$ to be decreasing such that $\psi(x)=1$ for $x\le 0$, $\psi(x)=0$ for $x\ge 1$, and $0\le -\psi'(x)\le 2$ for $x\in  \mathbb R$.
For $R>0$ fixed, we set $\psi_R(x)= \psi(x-R)$ for $x\in\R$ and define $f_R(z)=  f(z) \psi_R(|z|)$ for $z\in\C$. Since $f\in C^2(\C)\cap L^\infty$, it is obvious that $ f_R(z) \in C^2_c(\mathbb C)$,  the set of $C^2$ functions on $\mathbb R^2$   with compact support. From Theorem 2.1.1 of \cite{CS}, it follows that
$$
	f_R (z)= \fr 1{2\pi \mathrm i} \int_{\mathbb C} \frac{\frac{\partial f_R}{\partial \overline z}}{\xi-z} d\xi\wedge d\overline \xi.
$$
Notice that $ \frac{\partial f_R}{\partial \overline z} = \psi_R \frac{\partial  f}{\partial \overline z}+ f \frac{\partial  \psi_R}{\partial \overline z} $. By Lemma 2 on page 52 of \cite{Ah06}, we get
\begin{equation}\label{partial-estimate-z}
	\frac{\partial f_R}{\partial z}(z) 
	= \mathfrak{T}\left( \frac{\partial f_R}{\partial \overline z}\right)(z)
	= \mathfrak{T}\left(\psi_R \frac{\partial  f}{\partial \overline z}\right)(z)+ \mathfrak{T}\left(f \frac{\partial  \psi_R}{\partial \overline z}\right)(z).
\end{equation}
Now for $r>0$ and $|z|<r$, when $R$ is sufficiently large, it holds that
$$
	\Bigg | \mathfrak{T}\left(f \frac{\partial  \psi_R}{\partial \overline z}\right) \Bigg | (z) 
	\le \frac{\|f\|_{L^\infty} } {\pi(R-r)^2}  \int_{R\le |\xi|\le R+1}   dv(\xi)
	\le  \frac{3R \|f\|_{L^\infty} } { (R-r)^2},
$$
and hence
\begin{equation}\label{partial-deri-a}
	\left\|\mathfrak{T}\left(f \frac{\partial  \psi_R}{\partial \overline z}\right)\right\|_{L^s(D(0, r), dv)} 
	\le  \left\| \frac{\partial f}{\partial \overline {z}} \right\|_{L^s},
\end{equation}
where $D(0,r) = \{z\in\C : |z|<r\}$. In addition, by the boundedness of $\mathfrak{T}$ on $L^s$ (see, for example, the estimate (11) on page 53 in \cite{Ah06}), we get
\begin{equation}\label{partial-deri-b}
	\left\| \mathfrak{T}\left(\psi_R \frac{\partial  f}{\partial \overline z}\right)\right\|_{L^s}
	\le C \left\| \psi_R \frac{\partial f}{\partial \overline z} \right\|_{L^s} 
	\le C \left\|\fr {\partial  f}{\partial \overline z} \right\|_{L^s}.
\end{equation}
For $R$ sufficiently large, from \eqref{partial-estimate-z}, \eqref{partial-deri-a} and \eqref{partial-deri-b} it follows that
$$
	\left\|  \fr {\partial f}{\partial z} \right\|_{L^s(D(0, r), dv)}
	= \left\|  \fr {\partial f_R}{\partial z} \right\|_{L^s(D(0, r), dv)}
	\le C \left\|\fr {\partial  f}{\partial \overline z} \right\|_{L^s}.
$$
Therefore,
\begin{equation}\label{derivative-est}
	\left\|  \frac{\partial f}{\partial z} \right\|_{L^s} 
	\le C \left\|\fr {\partial  f}{\partial \overline z} \right\|_{L^s}.
\end{equation}

Now for $n\ge 2$ and $f\in L^\infty\cap C^2(\C^n)$, by \eqref{derivative-est}, we have
\begin{align*}
	\int_{\C^n} \left|\fr {\partial f}{\partial z_1}(\xi)\right|^s dv(\xi)
	&= \int_{{\mathbb C}^{n-1}} dv(\xi') \int_{\mathbb C} \left|  \frac
      {\partial f}{\partial z_1}(\xi_1, \xi')\right|^s dv(\xi_1) \\
      &\le C  \int_{{\mathbb C}^{n-1}} dv(\xi') \int_{\C} \left|
      \frac{\partial f}{\partial \overline z_1}(\xi_1, \xi')\right|^s dv(\xi_1).
\end{align*}
This implies \eqref{partial-dir} for $j=1$.  Similarly, \eqref{partial-dir} holds for $j=2, \ldots, n$, and the proof is complete.
 \end{proof}

\begin{proof}[Proof of the case $1\le q<p<\infty$]
Notice first that if $H_f\in \mathcal K(F^p(\varphi), L^q(\varphi)$, then by Theorem \ref{main1}, we have $f\in \mathrm{IDA}^{s, q}$ with $s=\fr {pq}{p-q}>1$. We use a decomposition $f=f_1+f_2$ as in \eqref{q-less-then-p-m} with $r=1$. Furthermore, by \eqref{f-j-bound-above}, we may assume that $\|f_1\|_{L^\infty}\le C \|f\|_{L^\infty}$. Then, from Lemma \ref{partial-derivatives} it follows that
$$
	\|\overline \partial\, \overline {f_1}\|_{L^s}
	\le  C  \sum_{j=1}^n  \left\| \fr {\partial \overline f}{\partial \overline z_j} \right\|_{L^s}
	= C  \sum_{j=1}^n  \left\| \fr {\partial   f}{\partial z_j} \right\|_{L^s}
	\le C  \sum_{j=1}^n  \left\| \frac{\partial   f}{\partial \overline z_j} \right\|_{L^s}
	\le C \|\overline \partial\,  {f_1}\|_{L^s}
$$
We also observe that $\|M_{q, r}(\overline {f_2})\|_{L^{s}} = \|M_{q, r}(f_2)\|_{L^{s}}<\infty$. Now Theorem \ref{BDA-equivalence} implies that $\overline f = \overline {f_1}+ \overline {f_2}\in  \textrm{IDA}^{s, q}$, and hence, by Theorem \ref{main1}, we get $H_{\overline f}\in \mathcal K(F^p(\varphi), L^q(\varphi))$.
\end{proof}

\begin{remark}
Notice that it follows from the preceding proof that
$$
	\|  H_{\overline f}\|_{F^p(\varphi)\to L^q(\varphi)}
	\le C \|  H_{ f}\|_{F^p(\varphi)\to L^q(\varphi)}.
$$
\end{remark}

\section{Application to Berezin-Toeplitz quantization}\label{quantization}

As an application and further generalization of our results, we consider deformation quantization in the sense of Rieffel~\cite{Ri89, Ri90} and focus on one of its essential ingredients in the non-compact setting of $\C^n$ that involves the limit condition
$$
	\lim_{t\to 0} \left \|T^{(t)}_f T^{(t)}_g- T^{(t)}_{fg}\right \|_{{  F}^2_t(\varphi) \to {  F}^2_t(\varphi)} =0.
$$
Recently this and related questions were studied in~\cite{BC16, BCH18, F20}, which also provide further physical background and references for this type of quantization.

Recall that  $\varphi \in   C^2( {\mathbb C}^{n})$ is real valued and $\mathrm{Hess}_{\mathbb R}\varphi \simeq{\mathrm E}$, where ${\mathrm E}$ is the $2n\times 2n$-unit matrix.
For $t>0$, we set
$$
	d\mu_t (z)= \frac1{t^n} \exp\left \{-  2\varphi\left( \fr    z { \sqrt{t}}\right)\right \} dv(z)
$$
and denote by $L^2_t(\varphi)$ the space of all Lebesgue measurable functions $f$ in $\C^n$ such that
$$
	\left\|f\right \|_{t} = \left\{ \int_{\C^n} \left| f \right|^2 d\mu_t(z) \right\}^{\frac12}.
$$
Further, we let $F^2_t(\varphi) = L^2_t(\varphi) \cap H(\Cn)$. Then clearly  $F^2_1(\varphi)=F^2(\varphi)$ and $L^2_1(\varphi)= L^2(\varphi)$  in terms of the spaces that were considered in the previous sections. Given $f\in L^\infty$,  we use the orthogonal projection $P^{(t)}$ from    $L^2_t(\varphi)$ onto $F^2_t(\varphi)$ to define the Toeplitz operator $T^{(t)}_f$ and the Hankel operator  $H^{(t)}_f$, respectively, by
$$
	T^{(t)}_f = P^{(t)} M_f \ \ \textrm{ and }\ \ 
	H^{(t)}_f= (\mathrm I- P^{(t)})M_f.
 $$
Let $ U_t$ be the  dilation acting on measurable functions in $\C^n$ as
$$
	U_t:  f \mapsto  f(\cdot  \sqrt{t}).
$$
It is easy to verify that  $U_t$ is a unitary operator from $L^2_t(\varphi)$ to $L^2(\varphi)$ (as well as a unitary operator from $F^2_t(\varphi)$ to $F^2(\varphi)$). Further, we have $U_t P^{(t)} U_t^{-1} = P^{(1)}$, which implies that
\begin{equation}\label{transform}
	U_t T^{(t)}_f U_t^{-1} 
	= T_{f(\cdot  \sqrt{t})},\ \  
	U_t H^{(t)}_f U_t^{-1} = H_{f(\cdot  \sqrt{t})}.
\end{equation}
Therefore,
\begin{equation}\label{u-operator-2}
	\| T^{(t)}_f\|_{{  F}^2_t(\varphi) \to {  F}^2_t(\varphi)} 
	= \| T_{f(\cdot \sqrt{t})}\|_{{  F}^2(\varphi) \to {  F}^2(\varphi)}
\end{equation}
and
\begin{equation}\label{u-operator-z}
	\| H^{(t)}_f\|_{{  F}^2_t(\varphi) \to {  L}^2_t(\varphi)} 
	= \| H_{f(\cdot \sqrt{t})}\|_{{  F}^2(\varphi) \to {  L}^2(\varphi)}.
\end{equation}

Given $f\in L^2_{\mathrm{loc}}$, for $z\in \Cn$ and $r>0$ set
$$
	MO_{2, r}(f)(z) 
	=  \left\{ \fr 1{|B(z, r)|} \int_{B(z, r)} \left| f-f_{B(z, r)}\right|^2 dv\right\}^{\frac12}
$$
where $f_{S} = \fr 1{|S|} \int_S f dv$ for $S\subset \C^n$ measurable.

The following definitions of $\BMO$ and $\VMO$ are analogous to the classical definition  introduced by John and Nirenberg \cite{JN61}, but they differ from those widely used in the study of Bergman and Fock spaces.

\begin{definition}\label{bmo}
We denote by $\mathrm{BMO}$ the set of all $f\in L^2_{\mathrm{loc}}$ such that
$$
	\left \|f \right\|_* = \sup_{z\in \Cn, \, r>0} MO_{2, r}(f)(z) <\infty
$$
and by $\mathrm{VMO}$ the set of all $f\in \BMO$ such that
$$
	\lim_{r\to 0} \sup_{z\in \Cn} MO_{2, r}(f)(z)=0.
$$
\end{definition}

\begin{definition}\label{bda}
We define $\mathrm{BDA}_*$ to be the family of all $f\in L^2_{\mathrm{loc}}$ such that
$$
	\|f\|_{\mathrm{BDA}_*}= \sup_{z\in \Cn, r>0}  G_{2, r}(f)(z)<\infty
$$
and $\mathrm{VDA}_*$ to be the subspace of all $f\in \mathrm{BDA}_*$ such that
$$
	\lim_{r\to 0} \sup_{z\in \Cn } G_{2, r}(f)(z)=0.
$$
\end{definition}

Given a family $X$ of functions on $\Cn$, we set $\overline X =\{\bar{f}:  f\in X\}$.

\begin{proposition}\label{bmo-bda}
It holds that
$$
	\mathrm{BMO}  = {\mathrm{BDA}_*} \cap \overline {\mathrm{BDA}_*}
	\quad{\rm and}\quad
	\mathrm{VMO}  = {\mathrm{VDA}_*} \cap \overline {\mathrm{VDA}_*}.
$$
Furthermore, we have
\begin{equation}\label{bmo-bda-a}
       \|f\|_{\mathrm{BMO}_*} \simeq \|f\|_{\mathrm{BDA}_*} + \|\overline f\|_{\mathrm{BDA}_*}
\end{equation}
for $f\in L^2_{\mathrm{loc}}$.
\end{proposition}

\begin{proof}
From a careful inspection of the proof of Proposition 2.5 in~\cite{HW18}, it follows that there is a constant $C>0$ such that, for $f\in L^2_{\mathrm{loc}}$ and $z\in \Cn$,  $r>0$,  there is a constant $c(z)$ for which
$$
	\left\{\frac1{|B(z, r)|}\int_{B(z, r)}\left|  f-c(z) \right|^2 dv \right\}^{\fr 12}
	\le C\left\{ G_{2, r}(f)(z) + G_{2, r}(\overline f)(z) \right\}.
$$
It is easy to verify that
$$
	MO_{2, r}(f)(z)
	\le \left\{   \fr 1{|B(z, r)|}\int_{B(z, r)}\left|  f-c(z) \right|^2 dv \right\}^{\frac12},
$$
and hence
$$
	MO_{2, r}(f)(z)\le C \left\{ G_{2, r}(f)(z)  +  G_{2, r}(\overline f)(z) \right\}.
$$
On the other hand, by definition, we have
$$
	G_{2, r}(f)(z)\le MO_{2, r}(f)(z).
$$
Thus, we have $C_1$ and $C_2$, independent of $f$, $r$ and $z$, such  that
\begin{equation}\label{bmo-bda-y}
\begin{split}
	C_1  \left\{ G_{2, r}(f)(z) + G_{2, r}(\overline f)(z) \right\}&\le MO_{2, r}(f)(z)\\
	&\le C_2\left\{ G_{2, r}(f)(z) + G_{2, r}(\overline f)(z) \right\}.
\end{split}
\end{equation}
Therefore, $f\in \mathrm{BMO}$ (or $f\in \mathrm{VMO}$) if and only if $f\in  {\mathrm{BDA}_*} \cap \overline {\mathrm{BDA}_*}$ (or $f\in  {\mathrm{VDA}_*} \cap \overline {\mathrm{VDA}_*}$).
The estimate in~\eqref{bmo-bda-a} follows from \eqref {bmo-bda-y}.
\end{proof}

\begin{theorem}\label{quantizaton-2}
Suppose $f\in L^\infty$. Then for all $g\in L^\infty$, it holds that
\begin{equation} \label{quantization-2-a}
	\lim_{t\to 0} \left \|T^{(t)}_f T^{(t)}_g- T^{(t)}_{fg}\right \|_{{  F}^2_t(\varphi) \to {  F}^2_t(\varphi)} = 0
\end{equation}
if and only if $f\in \overline {\mathrm{VDA}}_*$.
\end{theorem}

\begin{proof}
Given $f\in L^\infty$, it follows from \eqref{u-operator-z} that
$$
	\left  \| \left(H^{(t)}_{\overline f} \right )^* \right \|_{{  L}^2_t(\varphi)\to {  F}^2_t(\varphi)}
	= \left \| H^{(t)}_{\overline f} \right \|_{{  F}^2_t(\varphi) \to {  L}^2_t(\varphi)}
	= \left \| H_{f(\cdot \sqrt{t})}\right \|_{{  F}^2(\varphi) \to {  L}^2(\varphi)}.
$$
This and Theorem \ref{main1} imply
\begin{equation} \label{vda-3}
\begin{aligned}
	\frac1{C}  \|  G_{2, 1}(f(\cdot \sqrt{t}))\|_{L^\infty}
	&\le  \left \| \left(H^{(t)}_{\overline f}\right)^*\right \|_{{  L}^2_t(\varphi)\to {  F}^2_t(\varphi)}\\
	&\le  C \|  G_{2, 1}(f(\cdot \sqrt{t}))\|_{L^\infty},
\end{aligned}
\end{equation}
where the constant $C$ is independent of $f$ and $t$.

Suppose $f\in\overline {\mathrm{VDA}_*}$. Then, by definition, we have
$$
	\lim_{r\to 0}  \sup_{z\in \Cn } G_{2, r}(\overline f)(z) =0.
$$
It is easy to verify that
$$
	G_{2, 1} \left(f (\cdot \sqrt{t})\right)(z)
	= G_{2, \sqrt{t}} (f)\left(z\sqrt{t}\right).
$$
Now by \eqref{vda-3}, we get
\begin{equation} \label{vda-1}
	\lim_{t\to 0} \left \| \left(H^{(t)}_{\overline f}\right)^*\right \|_{{  L}^2_t(\varphi)\to {  F}^2_t(\varphi)}
	\le  C \,  \lim_{t\to 0}\|  G_{2, \sqrt{t}}(\overline  {f})\|_{L^\infty}=0.
\end{equation}
In addition, for $f, g\in L^\infty$, it is easy to verify that
\begin{eqnarray} \label{T-H-relation1}
	T^{(t)}_f T^{(t)}_g- T^{(t)}_{fg} 
	=-\left(H^{(t)}_{\overline f}\right)^* H^{(t)}_g.
\end{eqnarray}
Therefore, for all $g\in L^\infty$,
$$
	\lim_{t\to 0}\left \|  T^{(t)}_f T^{(t)}_g- T^{(t)}_{fg}\right \|_{{  F}^2_t(\varphi) \to {F}^2_t(\varphi)} 
	\le  \| g\|_{L^\infty}\,  \lim_{t\to 0} \left \| \left(H^{(t)}_{\overline f}\right)^*\right \|_{{  L}^2_t(\varphi)\to {  F}^2_t(\varphi)}=0,
$$
which gives \eqref{quantization-2-a}.

Conversely, suppose that (\ref{quantization-2-a}) holds for every $g\in L^{\infty}$. Let $g={\overline f}\in L^\infty$. Then it follows from~\eqref{T-H-relation1} that
\begin{align*}
	\lim_{t\to 0} \left \|H_{\overline f}^{(t)}\right \|^2_{{  F}^2_t(\varphi) \to {  L}^2_t(\varphi)} 
	&=\lim_{t\to 0} \left \|\left(H_{\overline f}^{(t)} \right)^* H_{\overline f}^{(t)}\right \|_{{  F}^2_t(\varphi) \to {  F}^2_t(\varphi)}\\
	&= \lim_{t\to 0} \|T^{(t)}_f T^{(t)}_{\overline f}- T^{(t)}_{|f|^2}\|_{{  F}^2_t(\varphi) \to {  F}^2_t(\varphi)}  =0.
\end{align*}
This and \eqref{vda-3} imply that $f\in\overline {\mathrm{VDA}_*}$.
\end{proof}

Combining Proposition \ref{bmo-bda} with Theorem \ref{quantizaton-2}, we obtain the following corollary, which is the main result of \cite{BCH18} when $\varphi(z)= \frac 1 8 |z|^2$.

\begin{corollary}\label{T-VMO}
Suppose $f\in L^\infty$. Then for all $g\in L^\infty$, it holds that
\begin{equation}\label{T-VMO-a}
        \lim_{t\to 0} \left \|T^{(t)}_f T^{(t)}_g- T^{(t)}_{fg}\right \| =0
        \quad {\rm and}\quad
        \lim_{t\to 0}\left \| T^{(t)}_g T^{(t)}_f - T^{(t)}_{fg}\right \| =0
\end{equation}
if and only if $g\in \mathrm{VMO}$. Here $\|\cdot\| = \|\cdot\|_{{  F}^2_t(\varphi) \to {  F}^2_t(\varphi)}$.
\end{corollary}

\section{Further remarks}\label{Further remarks}

For $1\le p, q<\infty$, we have characterized those $f\in \mathcal S$ for which $H_f: F^p(\varphi) \to L^q(\varphi)$ is bounded (or compact). For small exponents $0<p<q<1$, we have proved that this characterization remains true for compactness when $f\in L^\infty$. We also note that when $p\le q$ and $q\ge 1$, boundedness and compactness of Hankel operators $H_f : F^p(\varphi) \to L^p(\varphi)$ depend on $q$ (see Remark~\ref{proper-inclusions} and Theorem~\ref{main1}) while for $p>q$ we cannot say the same---we note that we have no statement analogous to Remark~\ref{proper-inclusions} for $\IDA^{s,q}$.

Moreover, for harmonic symbols $f\in \mathcal S$ and  $0<p, q<\infty$,  using the  Hardy-Littlewood theorem on the sub-mean value (see Lemma~2.1 of~\cite{HPZ07}, for example), we are able to characterize boundedness of $H_f : F^p(\varphi) \to L^q(\varphi)$ with the space $\mathrm{IDA}^{s, q}$. We will return to this topic in a future publication.

We also note that the space $F^\infty(\varphi)$ does not appear in our results because $\Gamma$ is not dense in it. Instead, it may be possible to consider the space
$$
	f^\infty(\varphi) = \{ f\in F^\infty(\varphi) : fe^{-\varphi}\in C_0(\C^n)\},
$$
which can be viewed as the closure of $\Gamma$ in $F^\infty(\varphi)$, and extend our results to this setting.

Regarding weights, the Fock spaces studied in this paper are defined with weights $\varphi\in C(\Cn)$ satisfying $\textrm {Hess}_{\mathbb R}\varphi \simeq {\mathrm E}$. As stated in Section~\ref{weighted Fock spaces}, these weights are contained in the class considered in \cite{SV12}. Now, we note that for the weights $\varphi$ in~\cite{SV12}, ${\mathrm i}\partial \overline \partial \varphi \simeq \omega_0$,  and from H\"{o}rmander's theorem on the canonical solution to $\overline \partial$-equation it follows that
 $$
    \|H_f g\|_{2, \varphi}^2 \le  \int_{\Cn} | g \overline \partial f |^2_{{\mathrm i}\partial \overline \partial} e^{-2\varphi} dv \le C\left \|g \, |\overline \partial f|\right\|_{2, \varphi}^2,
 $$
and hence we know that the conclusions of Theorem~\ref{main1} remain true when $q=2$ (see Theorem~4.3 of~\cite{HV22}). Upon these observations, we raise the following conjecture.

 \begin{conjecture} \label{conjecture}
 Suppose $\varphi\in C^2(\Cn)$ satisfying ${\mathrm i}\partial \overline \partial \varphi \simeq \omega_0$.  Then for $f\in \mathcal S$ and  $0< p, q<\infty$, $H_f\in \mathcal B(F^p(\varphi), L^q(\varphi))$ if and only if $f\in \mathrm{IDA}^{s, q}$ where $s=\fr {pq}{p-q}$ if $p>q$ and $s=\infty$ if $p\le q$.
 \end{conjecture}

In the literature, there are a number of interesting results on the simultaneous boundedness (and compactness) of Hankel operators $H_f$ and $H_{\overline f}$. These types of characterizations often involve the function spaces $\textrm{BMO}^q $ and $ \textrm{IMO}^{s,q} $ in their conditions; see, e.g., \cite{HW18, Zh12} and the references therein.
For   $1\le q<\infty$ and $1\le s\le \infty$,  set $\overline { \textrm{IDA}^{s,q}}= \{\overline f: f\in \textrm{IDA}^{s,q}\}$. Then Proposition~2.5 of~\cite{HW18} shows that $\textrm{IDA}^{s,q}\cap\overline { \textrm{IDA}^{s,q}} = \textrm{IMO}^{s, q}$ and the results of Section~\ref{hankel-A} provide a description of the simultaneous boundedness (or compactness) of $H_f$ and $  H_{\overline f}$ as seen in the following theorem, where as before, we set $s=\fr {pq}{p-q}$ if $p>q$ and $s=\infty$ if $p\le q$.

\begin{theorem}\label{doubling-bdd}
Let $\varphi\in C^2(\Cn)$ be  real valued, ${\mathrm {Hess}}_{\mathbb R}\varphi \simeq {\mathrm E}$, and let $f\in \mathcal S$.  For  $1\le  p, q<\infty$, Hankel operators  $H_f$ and $  H_{\overline f}$ are  simultaneously bounded from $F^p(\varphi)$ to $ L^q(\varphi))$ if and only if $f\in \mathrm{IMO}^{s, q}$.
\end{theorem}

We state one more conjecture related to Theorem~\ref{main2}, in which we proved that for $f\in L^\infty$ and $0<p< \infty$, $H_f$ is compact on $F^p(\varphi)$ if and only if $H_{\overline  f}$ in compact on $F^p(\varphi)$. Recall that this phenomenon does not occur for Hankel operators on the Bergman space or on the Hardy space. As predicted by Zhu~\cite{Zh12}, and verified for Hankel operators on the weighted Fock spaces $F^p(\alpha)$ with $1<p<\infty$ in~\cite{HV19}, a partial explanation for this difference is the lack of bounded holomorphic or harmonic functions on the entire complex plane.  From this point of view it is natural to suggest that a similar result should remain true for Hankel operators mapping from $F^p(\varphi)$ to $L^q(\varphi)$.

\begin{conjecture}\label{conjecture-2}
Suppose that $\varphi\in C^2(\C^n)$ satisfies ${\mathrm i}\partial \overline \partial \varphi \simeq \omega_0$ and $0<p, q<\infty$. Then for $f\in L^\infty$,  $H_{ {f} }\in \mathcal K \left(F^p(\varphi), L^q(\varphi)\right)$ if and only if $H_{\overline {f}}\in \mathcal K\left(F^p(\varphi), L^q(\varphi)\right)$.
\end{conjecture}

Notice that $\IDA^{s,q}\cap L^\infty$ is a Banach algebra under the norm $\|\cdot\|_{\IDA^{s,q}} + \|\cdot\|_\infty$. We can also express Conjecture~\ref{conjecture-2} in algebraic terms, that is, we conjecture that $\IDA^{s, q}\cap L^\infty$ on $\Cn$ is closed under the conjugate operation $f \mapsto \overline f$, where $1<s\le \infty$ and $0<q<\infty$.

\bigskip
Related to our work on quantization and Theorem~\ref{quantizaton-2} in particular, we conclude this section with the following problem: Characterize those $f\in L^\infty$ for which it holds that
$$
	\lim_{t\to 0} \|T^{(t)}_fT^{(t)}_g - T^{(t)}_{fg}\|_{S_2} = 0
$$
for all $g\in L^\infty$, where $\|\cdot\|_{S_2}$ stands for the Hilbert-Schmidt norm. It would also be important to consider this question for other Schatten classes $S_p$.

\section*{Acknowledgments}

\noindent The authors thank the referee for valuable comments and suggestions.

\smallskip

Z.~Hu was partially supported by the National Natural Science Foundation of China (12071130, 12171150). J.~Virtanen was supported in part by Engineering and Physical Sciences Research Council (EPSRC) grant EP/T008636/1.\\

\end{document}